\documentclass{amsart}

\newtheorem{tm}{Theorem}[section]
\newtheorem{ap}{Assumption}[section]

\newtheorem{prop}{Proposition}[section]
\newtheorem{lm}{Lemma}[section]
\newtheorem{cor}{Corollary}[section]

\theoremstyle{remark}
\newtheorem{rk}{Remark}[section]

\numberwithin{equation}{section}

\usepackage{graphicx} 
\usepackage{extarrows}
\usepackage{latexsym}
\usepackage{amsmath}
\usepackage{amssymb}
\usepackage{amsfonts}
\usepackage{verbatim}
\usepackage{mathrsfs}
\usepackage{color}
\usepackage{xcolor}
\usepackage[colorlinks,citecolor=blue,urlcolor=blue]{hyperref}
\usepackage{soul}

\newcommand{\ee}{\mathbb E}

\newcommand{\pp}{\mathbb P}
\newcommand{\nn}{\mathbb N}
\newcommand{\rr}{\mathbb R}

\newcommand{\CC}{\mathcal C}
\newcommand{\DD}{\mathcal D}
\newcommand{\LL}{\mathcal L}

\newcommand{\PP}{\mathcal P}

\newcommand{\OOO}{\mathscr O}
\newcommand{\FFF}{\mathscr F}

\newcommand{\<}{\langle}
\renewcommand{\>}{\rangle}
\allowdisplaybreaks \allowdisplaybreaks[4]

\newcommand{\dd}{\mathrm{d}}

\begin{document}

\title[Uniform Weak and Ergodic Error Estimates of Superlinear SPDEs]
    {Uniform-in-Time Weak and Ergodic Error Estimates of a Nonlinearity-Explicit Full Discretization for Superlinear SPDEs Driven by Multiplicative Noise}

\author{Jingjing CAI}
\address{Department of Mathematics, Southern University of Science and Technology, Shenzhen, 518055, P. R. China}
\email{12431001@mail.sustech.edu.cn}
 
\author{Zhihui LIU}
\address{Department of Mathematics \& National Center for Applied Mathematics Shenzhen (NCAMS) \& Shenzhen International Center for Mathematics, Southern University of Science and Technology, Shenzhen, 518055, P. R. China}
\email{liuzh3@sustech.edu.cn (Corresponding author)}

\author{Xiaoming WU}
\address{Department of Mathematics, Southern University of Science and Technology, Shenzhen, 518055, P. R. China}
\email{12331004@mail.sustech.edu.cn}

\thanks{The authors are supported by the National Natural Science Foundation of China (NNSFC), No. 12101296, Basic and Applied Basic Research Foundation of Guangdong Province, No. 2024A1515012348, and Shenzhen Basic Research Special Project (Natural Science Foundation) Basic Research (General Project), No. JCYJ20240813094919026.}


\date{}
 

\begin{abstract}
For a class of superlinear SPDEs driven by multiplicative noise, we prove an (essentially) sharp uniform-in-time (UIT) weak convergence rate for the nonlinearity-explicit Galerkin tamed Euler method (GTEM). Under standard monotonicity assumptions, the proof combines Malliavin calculus with regularity theory for the associated backward Kolmogorov equation (BKE), leading to UIT moment, H\"older, and Malliavin estimates, along with regularity estimates for the BKE solution.  
These estimates, together with a weak error decomposition and Malliavin integration by parts (IBP) formula, then yield a UIT weak convergence rate $\tau^\rho+\lambda_N^{-(\rho+\gamma/2)}$ for any $\rho \in (0,1)$, where $\gamma\in[0,1)$ quantifies the assumed spatial Sobolev regularity. 
Consequently, we obtain a sharp ergodic error estimate between the exact and numerical invariant measures. Numerical experiments support the theory.
\end{abstract}

\maketitle

\section{Introduction}\label{sec1}

Under random forcing, SPDEs are standard models for spatially extended systems, including stochastic reaction-diffusion and phase-field models such as the stochastic Allen--Cahn equation (SACE); see \cite{DZ14, Cerrai01}. In their long-time analysis, the ergodic invariant measure is central: it characterizes stationary statistical behavior and yields stationary expectations of observables \cite{DZ96}. This motivates the demand for numerical methods that accurately reproduce both finite-time transient dynamics and long-time statistical quantities, including expectations and invariant measures, with UIT error bounds.

Strong and weak approximations of SPDEs with globally Lipschitz coefficients are by now well developed; see, e.g., \cite{DP09, Debussche11, AL16, ACQS20, BD18}. The situation changes substantially when the drift is only monotone and polynomially growing. In that case, stability and moment bounds are no longer automatic; implicit, splitting, truncated, and tamed methods have been developed to recover finite-time convergence; see \cite{BJ19, JP20, BCH19, CH19, BGJK23, Wang20, WW25} and references therein. Long-time weak approximation is more demanding still: one needs numerical moment estimates, Kolmogorov regularity, and accumulated local weak-error bounds that are uniform over an infinite time horizon.

The long-time weak theory for parabolic SPDEs with superlinear drift has been largely confined to additive noise. Cui, Hong, and Sun \cite{CHS21} pioneered a  drift-implicit Euler--Galerkin (DIEG) scheme, proving UIT BKE regularity, UIT weak convergence, and invariant-measure approximation. Subsequent work extended these results to nonlinearity-explicit tamed discretizations \cite{Brehier22, JW25, WangCao26}; see also \cite{QiWang26} for a finite-element treatment and \cite{Brehier24, Brehier25} for total-variation estimates. While these works establish essentially sharp long-time weak convergence in the superlinear regime, they all hinge critically on the diffusion coefficient being independent of the solution. 

On the other hand, recent work has advanced finite-time weak approximation under multiplicative noise. Breit and Prohl \cite{BP24} proved first-order temporal weak convergence for a drift-implicit Euler-type scheme applied to the SACE with multiplicative colored noise; see also \cite{ZZZC26} for the DIEG scheme under multiplicative trace-class noise. However, the question of UIT weak errors and stationary weak bias remains open.

A different line of work has developed long-time stability, ergodicity, and strong approximation for superlinear SPDEs driven by multiplicative noise.  Numerical ergodicity and UIT strong estimates for DIEG schemes were obtained in \cite{Liu25, Liu26}; Liu and Shen \cite{LS25} introduced Lyapunov-preserving tamed-FEM, proved unique ergodicity and optimal strong convergence rates.  The recent tamed finite element analysis in \cite{CL26} further provides UIT strong error bounds and a quantitative approximation of invariant measures in Wasserstein distances. However, these results do not address the more demanding problem of UIT higher-order weak error estimates. To our knowledge, a UIT weak convergence result, along with a corresponding ergodic error estimate, has not yet been obtained for any numerical approximation of equations with both superlinear drift and multiplicative noise.

In this paper, we consider the second-order parabolic SPDE
\begin{align}\label{spde} 
& {\rm d} X(t, \xi)  
=[\Delta X(t, \xi)+f(X(t, \xi))] {\rm d}t
+g(X(t, \xi)) {\rm d}W(t, \xi),\quad (t, \xi) \in \rr_+\times \OOO,
\end{align}
with the homogeneous Dirichlet boundary condition (DBC) and initial condition $X_0$. Here, $\OOO \subset \rr^d$ ($d=1,2,3$) is a bounded open set with a piecewise smooth boundary $\partial \OOO$, $f: \rr \to \rr$ is monotone and polynomially growing, while $g:\rr \to \rr$ is Lipschitz in the standard infinite-dimensional sense, and $W$ is an infinite-dimensional Wiener process (cf. Section \ref{sec2}). The SACE -- a model for phase transitions under stochastic perturbations -- corresponds to the choice $f(\xi)=\epsilon^{-2}(\xi-\xi^3)$, with $\epsilon>0$ denoting the interface thickness; see \cite{BGJK23, BP24, CH19, FLZ17, HS23, JW25, LL26, Liu22, LQ20, LQ21} and references therein. 

The main difficulty stems from the interplay between the superlinear drift and the state-dependent diffusion.  The drift forces UIT Sobolev moment bounds on both the Galerkin solution and the GTEM. The state dependence of the diffusion, however, introduces extra stochastic terms into the first three variational equations, necessitating estimates for the corresponding derivatives of the BKE solution that handle both short-time singularities and long-time decay. Moreover, the local weak defects include stochastic-history terms whose discrete smoothing is not directly integrable. We treat these via Malliavin IBP formula; the diffusion defect, in particular, requires the third derivative of the BKE solution -- a feature with no analog in the additive-noise case.

Our main result, Theorem~\ref{tm-weak}, establishes an essentially optimal UIT weak error bound. Under an explicit sufficient dissipativity condition and the spatial regularity assumptions with parameter $\gamma\in[0,1)$ specified in Section~\ref{sec2}, for every constant $\rho \in (0,1)$ and test function $\varphi \in \CC_b^3(H)$, the error satisfies
\begin{align*}
|\ee [\varphi(X(t_m))-\varphi(X^N_m)]|
\le C (1+t_m^{-\rho}) (\tau^\rho+\lambda_N^{-(\rho+\gamma/2)}), \quad m \ge 2,
\end{align*}
where $C$ is independent of $m$, $N$, and $\tau$.  
Since $\rho$ can be chosen arbitrarily close to one, the temporal and spatial weak orders are essentially $1$ and $1+\gamma/2$, respectively; when $\gamma=0$, these are essentially twice the corresponding strong orders obtained in \cite{CL26}.
The same rate is inherited by regular observables of the invariant measures in the limit $t_m\to \infty$; see Corollary~\ref{cor-erg}.
 We note that the UIT weak and ergodic error estimates for the weak dissipativity case with nondegenerate multiplicative noise will be investigated in a separate paper, as we need to develop tools distinct from those of the present study.

The proof relies on a Kolmogorov--Malliavin strategy. We first derive UIT Sobolev moment, H\"older, and Malliavin estimates for the Galerkin solution, the GTEM, and its continuous interpolation. We then analyze the BKE associated with an auxiliary Galerkin system, obtaining UIT estimates for its first three derivatives from the variational equations.  
For the second and third derivatives, we establish mixed Sobolev estimates with one argument measured in $\dot H^\chi$. In $d=3$, choosing $\chi>0$ compensates for the dimension-dependent loss in the nonlinear product estimates and preserves the integrability of the short-time singularities; this is a key ingredient in our derivation of an essentially first-order temporal weak rate.
It\^o formula applied to the interpolation yields a telescoping weak-error representation; its local defects split into linear, drift, and diffusion components. The singular stochastic-history terms are handled via Malliavin IBP, and exponentially weighted discrete convolutions ensure the accumulated local error is UIT. These arguments extend the long-time weak convergence theory from additive to multiplicative noise for a nonlinearity-explicit tamed scheme. 

The rest of the paper is organized as follows. Section~\ref{sec2} introduces the setting, assumptions, numerical scheme, and main results. The uniform moment, H\"older, Malliavin, and Kolmogorov estimates are established in Section~\ref{sec3}. In Section~\ref{sec4}, we prove the UIT weak and ergodic error estimates. Section~\ref{sec5} presents numerical experiments, and the appendices collect auxiliary estimates.

\section{Preliminaries and Main Results}
\label{sec2}
 
\subsection{Notations}

In what follows, for $p \in [1, \infty]$, the usual real-valued Lebesgue spaces in $\OOO$, $\rr_+:=[0, \infty)$, and $\Omega$ will be denoted by $(L_\xi^p, \|\cdot\|_{L_\xi^p})$, $(L^p_t, \|\cdot\|_{L^p_t})$, and $(L_\omega^p, \|\cdot\|_{L_\omega^p})$, respectively. Mixed norms, such as $\|\cdot\|_{L^p_\omega L^{p_1}_t L_\xi^{p_2}}$, will also be used, with the order of the norms ordered according to the needs of the estimate.

Let $H:= \{u \in L_\xi^2 : u|_{\partial\OOO}=0\}$, with norm $\|\cdot\|$ (or $\|\cdot\|_{L_\xi^2}$) and inner product $\<\cdot,\cdot\>$. Let $\CC_b^3(H)$ be the space of three times continuously Fr\'echet differentiable functions on $H$ with bounded derivatives. Let $\LL:=\LL(H)$ denote bounded linear operators, with norm $\|\cdot\|_{\LL}$, and $A$ be the Dirichlet Laplacian on $H$. Then $-A$ has  eigenvalues $0<\lambda_1 \le \lambda_2 \le \cdots$ with eigenvectors $\{e_k\}_{k \in \nn_+}$ vanishing on $\partial \OOO$. Moreover,  $A$ is the infinitesimal generator of an analytic $\CC_0$-semigroup $S(t)=e^{A t}$, and one can define the fractional power $(-A)^{\theta/2}$, $\theta \in \rr$, with domain $\dot H^\theta$ equipped with the norm $\|u\|_\theta:=\|(-A)^{\theta/2} u\|$.
In particular, one has $\dot H^1=H_0^1 := \{u, \nabla u \in L_\xi^2 : u|_{\partial\OOO}=0\}$ with inner product $\<\cdot, \cdot\>_1:=\<\cdot, \cdot\>+\<\nabla \cdot, \nabla \cdot\>$, and its dual is $\dot H^{-1}$ with pairing ${_{-1}}\<\cdot, \cdot\>_1$. 
We will use the following well-known Poincar\'e inequalities:
\begin{align} \label{poin}
\|\nabla u\|^2 \ge \lambda_1 \|u\|^2, \quad u \in \dot H^1; \quad 
\|A v\|^2 \ge \lambda_1 \|\nabla v\|^2, \quad v \in \dot H^2.
\end{align} 
   
Let $U$ be another separable Hilbert space and $Q$ a self-adjoint and nonnegative definite trace-class operator on $U$.
Denote by $U_0:=Q^{1/2} U$ and $(\LL_2^\theta:=HS(U_0; \dot H^\theta), \|\cdot\|_{\LL_2^\theta})$ the space of Hilbert--Schmidt operators from $U_0$ to $\dot H^\theta$ for $\theta \in \rr_+$.
Let $W:=\{W(t):\ t \in \rr_+\}$ be a $U$-valued $Q$-Wiener process on $(\Omega,\FFF,(\FFF_t)_{t \in \rr_+},\pp)$, i.e., there exists an orthonormal basis $\{g_k\}_{k=1}^ \infty$ of $U$ diagonalizing $Q$ with the eigenvalues $\{q_k\}_{k=1}^ \infty$: $Q g_k= q_k g_k$, $k \in \nn_+$, and a sequence of mutually independent 1D Brownian motions $\{\beta_k\}_{k=1}^ \infty $ such that $W(t)=\sum_{k \in \nn_+} \sqrt{q_k} g_k \beta_k(t)$, $t \ge 0.$ 

Throughout, we use $C$, $c$, $c_1$, $\cdots$, to denote generic positive constants that are independent of various discrete parameters and may vary from line to line.

\subsection{Main assumptions}

In this subsection, we introduce the main assumptions on the coefficients in Eq. \eqref{spde}. 
Specifically, we impose the following coercivity, monotonicity, and polynomial growth conditions, as imposed in \cite{CL26, LS25}.

\begin{ap} \label{ap-f} 
$f: \rr \to \rr$ is three times continuously differentiable and there exist an integer $q\ge2$, constants $K_1,K_3 \in \rr$, $K_2>0$, and $L_i>0$, $i=1,\ldots,4$ such that for all $\xi \in \rr$ and $\tau \in (0, 1)$,  
\begin{align}  
 \xi f(\xi) \le K_1 - & K_2 |\xi|^{q+2},  \label{f-coe} \\ 
[1+\tau |\xi|^{2q}]f'(\xi) - q \tau |\xi|^{2(q-1)} \xi f(\xi) 
& \le K_3 (1+\tau|\xi|^{2q})^{3/2}, \label{f'} \\
|f(\xi)| \le L_1 + L_2 |\xi|^{q+1}, ~  
& |f'''(\xi)| \le L_3 + L_4 |\xi|^{q-2}. \label{f-grow}
\end{align}
\end{ap}

Throughout the paper, we assume that $q \ge 2$ when $d = 1,2$ and $q = 2$ when $d = 3$. We rely on the following Sobolev embeddings:
\begin{align}
    \dot H^1 & \hookrightarrow L^{2(q+1)}_\xi \hookrightarrow H \hookrightarrow \dot H^{-1} \hookrightarrow L^{\frac{2(q+1)}{2q+1}}_\xi, \quad d=1,2,3, \label{embedding} \\
    \dot H^{1+\gamma} & \hookrightarrow L^ \infty_\xi, \quad \text{where } \gamma = 0 \text{ if } d=1, \; 0<\gamma \text{ if } d=2, \text{ and } 1/2<\gamma \text{ if } d=3. \label{gamma}
\end{align}   
Then we can define the Nemytskii operator $F: \dot H^1  \to \dot H^{-1}$ and $G: H \rightarrow \LL_2^0 $ associated with $f$ and $g$ respectively,  by
\begin{align}  
F(u)(\xi):=f(u(\xi)),  &\quad u \in \dot H^1, ~ \xi \in \OOO,\label{df-F}\\
G(u)g_k(\xi):=g(u(\xi))g_k(\xi), & \quad u \in H, ~ k \in \nn_+,~ \xi \in \OOO.\label{df-G}
\end{align} 
With these definitions, Eq. \eqref{spde} can be written as the infinite-dimensional stochastic evolution equation
\begin{align} \label{see}
    {\rm d}X(t) = [A X(t) + F(X(t))] {\rm d}t + G(X(t)) {\rm d}W(t), \quad t \in \rr_+.
\end{align} 

To handle the superlinear growth allowed by Assumption~\ref{ap-f}, we follow \cite{LS25} and define the easily-formulated tamed function
 \begin{align} \label{f-tau}
 f_\tau (\xi):=\frac{f(\xi)}{(1+\tau|\xi|^{2q})^{1/2}}, \quad \xi \in \rr,
\end{align} 
where $\tau \in (0, 1)$ denotes the temporal step-size. 

We will frequently use the following facts about $f_\tau$ defined in \eqref{f-tau}, $F$ defined in \eqref{df-F}, and the corresponding tamed Nemytskii operator $F_\tau$ of $f_\tau$.

\begin{rk} 
Under Assumption \ref{ap-f}, \cite[Lemma 3.2]{CL26} showed that there exist constants $\tau_0 \in (0,1)$ and $T_1>0$ such that, for any $\tau \in (0,\tau_0]$, 
\begin{align} \label{ftau'+}
f_\tau'(\xi)+ \tau|f_\tau'(\xi)|^2 & \le K_3+K_3^2\tau, \quad \xi \in \rr,\\
\< w, F_\tau(w)\>_1 + \tau \|F_\tau(w)\|_1^2
& \le T_1+(K_3+K_3^2\tau) \|w\|_1^2, \quad w \in \dot H^1. \label{dis-mon} 
\end{align}   
\end{rk}

\begin{rk}\label{rk-F-bounds}
Under Assumption~\ref{ap-f}, we have
\begin{align}
|f_\tau'(\xi)| & \le C\tau^{-1/2}, \quad \xi \in \rr, \label{ftau'} \\
\tau\|F_\tau(u)\|_1^2 & \le C(1+\|u\|_1^2), \quad u \in \dot H^1, \label{tauFtau1} \\
\|F(u)-F_\tau(u)\| & \le C\tau (1+\|u\|_{L_\xi^{2(3q+1)}}^{3q+1}), \quad u \in L_\xi^{2(3q+1)}, \label{f-ftau} \\
\|F(u)\|_1+\|F_\tau(u)\|_1 & \le C (1+\|u\|_{L_\xi^ \infty}^q)(1+\|u\|_1),
\quad u \in \dot H^1\cap L_\xi^ \infty, \label{F-bounds-0}\\
\|DF(v)u\| & \le C (1+\|v\|_{L_\xi^ \infty}^q)\|u\|,
\quad v \in L_\xi^ \infty, ~u \in H. \label{F-bounds-1}
\end{align}
Furthermore, for $\varrho \in (1/4,1/2)$ and $w \in L_\xi^ \infty$, we have
\begin{align}
\|D^2F(w)(u_1,u_2)\|_{-2\varrho}
&\le C(1+\|w\|_{L_\xi^ \infty}^{q-1})
\|u_1\|_{\mu_1}\|u_2\|_{\mu_2},
\label{D2F-d}\\
\|D^3F(w)(u_1,u_2,u_3)\|_{-2\varrho}
&\le C(1+\|w\|_{L_\xi^ \infty}^{q-2})
\|u_1\|_{\nu_1}\|u_2\|_{\nu_2}\|u_3\|_{\nu_3},
\label{D3F-d}
\end{align}
provided $\mu_j,\nu_j \in [0,\min\{1, d/2\})$ with $\mu_1+\mu_2+2\varrho>d/2,~ \nu_1+\nu_2+\nu_3+\min\{2\varrho,\frac d 2\}>d$. 
The proofs of these estimates are given in Appendix~\ref{app-A}.

\end{rk}

Our main conditions on the diffusion operator $G: H \to \LL_2^0$ are as follows.

\begin{ap}\label{ap-G} 
\begin{enumerate}
\item[(1)]
$G$ is three times continuously differentiable and there exist positive constants $K_i$, $i=4,  \cdots, 8$ such that for any $u, h, k, l \in H$ and $v \in \dot H^1$, 
\begin{align}  
    \|DG(u) h\|_{\LL_2^0} & \le K_4 \|h\|, \label{DG} \\ 
    \|G(v) \|^2_{\LL_2^1} & \le K_5 + K_6 \|v\|^2_1,  \label{G1} \\
    \|D^2G(u) (h,k)\|_{\LL_2^0} & \le K_7 \|h\|\|k\|, \label{D2G}\\ 
    \|D^3G(u) (h,k,l)\|_{\LL_2^0} & \le K_8 \|h\|\|k\|\|l\|. \label{D3G} 
\end{align}   
\item[(2)]  
$G:\dot H^{1+\gamma}\longrightarrow \LL_2^{1+\gamma}$ grows linearly, i.e., there exist positive constants $K_9$ and $K_{10}$ such that  
\begin{align}
\|G(z)\|_{\LL_2^{1+\gamma}}
\le K_9+K_{10}\|z\|_{1+\gamma},
\qquad z\in\dot H^{1+\gamma}.
\label{G2}
\end{align} 
\end{enumerate}
\end{ap}

\begin{rk}
    As a byproduct of \eqref{DG}, $G$ is Lipschitz and grows linearly:
\begin{align} \label{G-lip}
    \|G(u)-G(v)\|_{\LL_2^0} \le K_4 \|u-v\|, \quad 
    \|G(u)\|_{\LL_2^0} \le L_5 (1+\|u\|), \quad  u,v \in H,
\end{align}
where $L_5$ is a positive constant depending only on $K_4$ and $G(0)$.
\end{rk}

\begin{rk} \label{ap-str}
Under \eqref{f'} and \eqref{DG} with $\lambda_1>K_3+\frac{2p^*-1}{2}K_4^2$ for a constant $p^* > 1$, the following coupled monotonicity condition holds with the positive constant $K:=\lambda_1-K_3-\frac{2p^*-1}{2}K_4^2$:
\begin{align} \label{D-coup}
_{-1} \<[A +D F(y)]x, x\>_1  + \frac{2p^*-1}{2} \|DG(y)x\|_{\LL_2^0}^2 
\le - K \|x\|^2, \quad x, y \in \dot H^1.
\end{align} 
\end{rk}

\subsection{GTEM}

To introduce the fully discrete scheme, let $N \in \nn_+$ and define $V_N:= \text{span}\{e_1, e_2, \dots, e_N\}$ as the finite-dimensional space spanned by the first $N$ eigenfunctions of $-A$.
We define the spectral Galerkin approximate Laplacian operator $A_N: V_N \to V_N$ and the generalized orthogonal projection operator $\PP_N: \dot H^{-1} \to V_N$, respectively, as 
\begin{align*}  
\<A_N u^N, v^N\> =-\<\nabla u^N, \nabla v^N\>, \quad 
\<\PP_N u, v^N\> & =_{-1}\<u, v^N\>_1,
\quad u^N, v^N \in V_N, ~ u \in \dot H^{-1}. 
\end{align*}  
The spectral Galerkin approximation of \eqref{see} is the following finite-dimensional stochastic differential equation on $V_N$, with $X^N(0)=\PP_N X_0$:
\begin{align}\label{spe}  
	\dd X^N(t)&=(A_N X^N(t)+\PP_N F(X^N(t)))\dd t+\PP_N G(X^N(t))\dd W(t),\quad t>0. 
\end{align}

Following \cite{LS25}, we replace $F$ in \eqref{spe} with the tamed drift $F_\tau$ to obtain a fully discrete scheme that is explicit in the nonlinearity. 
The resulting fully discrete method, referred to as the GTEM, seeks an $\{\FFF_{t_j}: ~j \in \nn\}$-adapted, $V_N$-valued discrete process $\{X^N_j: ~j \in \nn\}$ such that 
\begin{align}\label{gtem}
	X^N_{j+1}
=X^N_j+\tau A_N X^N_{j+1}
+\tau \PP_N F_\tau(X^N_j)
+\PP_N G(X^N_j) \delta_j W,
\end{align}
where $\delta_j W:=W(t_{j+1})-W(t_j),~ t_j=j \tau,~ j \in \nn$, with initial datum $X^N_0:=\PP_N X_0$. 
For $s\ge0$, define $ \ell(s):=\max\{j \in \mathbb N: t_j\le s\}$ and $[s]_\tau:=t_{\ell(s)}$.

As the scheme \eqref{gtem} is only linearly implicit, it can be solved pathwise uniquely. 
Moreover, $\{X^N_j, \FFF_{t_j}\}_{j \in \nn}$ is a (time-homogeneous) Markov chain. 
It is clear that the GTEM \eqref{gtem} is equivalent to the scheme
\begin{align}\label{full}
X^N_{j+1}=S_{N,\tau} X^N_j+\tau S_{N,\tau} \PP_N F_\tau(X^N_j)
+ S_{N,\tau} \PP_N G(X^N_j)\delta_j W,
\quad j \in \nn,
\end{align}
where $S_{N,\tau}:=({\rm Id}-\tau A_N)^{-1}$, with ${\rm Id}$ denoting the identity operator in $V_N$, is a space-time approximation of the continuous semigroup $\{S(t)=e^{A t}: t \ge 0\}$ in one step. 
By iteration, we get 
\begin{align}\label{full-sum}
X^N_j
=S_{N, \tau}^jX^N_0+\tau \sum_{i=0}^{j-1} S_{N, \tau}^{j-i} \PP_N F_\tau(X^N_i) +\sum_{i=0}^{j-1} S_{N, \tau}^{j-i} \PP_N G(X^N_i) \delta_i W,
\quad j \in \nn_+.
\end{align} 

We construct a continuous interpolation of the GTEM \eqref{gtem} or its equivalent representation \eqref{full-sum}: for $j \in \nn$, $\hat X^N(t_j)=X^N_j$, and $t \in [t_j, t_{j+1}]$, 
\begin{align} \label{xhat}
	\dd \hat X^N(t)=[A_N S_{N,\tau} X^N_j+S_{N,\tau} \PP_N F_\tau(X^N_j)] \dd t+S_{N,\tau}\PP_NG(X^N_j)\dd W(t).
\end{align}

\subsection{Main results}

Our main result is the following UIT weak error estimate for the GTEM \eqref{gtem} applied to \eqref{see}.
Throughout, without loss of generality, we assume that the initial datum $X_0$ is a deterministic element; all results remain valid provided $X_0 \sim \FFF_0$ possesses finite $p$-moments for sufficiently large $p \ge 2$. 

To formulate a convenient sufficient dissipativity condition, we set
\begin{align*}
\Lambda_{\rm w}:=\max \Big\{ \frac{4q+3}{2}K_4^2,\frac{15-4/q}{2}K_4^2, \Big((8q-2)(q+1)^2-\frac12 \Big)K_6 \Big\}.
\end{align*}
The first, second, and third terms arise from the positive-order Malliavin estimates, the third-order variational estimates, and the highest-order Sobolev moment estimate, respectively. Thus, $\Lambda_{\rm w}$ is a uniform sufficient threshold for $d=1,2,3$ and is not claimed to be the weakest condition in each dimension.

\begin{tm}\label{tm-weak} 
Let $X_0 \in \dot H^{1+\gamma}$ with $\gamma \in [0,1)$ satisfying \eqref{gamma} and Assumptions~\ref{ap-f}-\ref{ap-G} hold with $\lambda_1 > K_3 +\Lambda_{\rm w}$.
For any $\rho \in (0,1)$ and $\varphi \in \CC_b^3(H)$, there exist constants $C$ and $\tau_{\max} \in (0,1)$ such that for any $\tau \in (0, \tau_{\max}]$, $N \in \nn_+$, and $m \ge 2$,
\begin{align} \label{weak}
|\ee [\varphi(X(t_m))-\varphi(X^N_m)]|
\le C(1+\|X_0\|_{1+\gamma}^{(8q-1)(q+1)^2}) (1+t_m^{-\rho}) (\tau^\rho+\lambda_N^{-(\rho+\gamma/2)}).
\end{align}
\end{tm}

We point out that, when $\gamma=0$, the above weak convergence rate \eqref{weak} is essentially sharp, being twice the strong convergence rate $\mathcal O(\tau^{1/2}+\lambda_N^{-1/2})$ established in \cite[Theorem~4.1]{CL26} under the same-type conditions as in Theorem~\ref{tm-weak}; when $\gamma>0$, the additional spatial regularity improves the spatial weak order by $\gamma/2$. 
 
Under the dissipativity condition in Theorem~\ref{tm-weak}, the exact equation and the GTEM are exponentially mixing and possess unique invariant measures $\pi$ on $H$ and $\pi_\tau^N$ on $V_N$, respectively; see \cite[Theorem~2.1]{Liu26} and \cite[Lemma~5.1 and Theorem~5.1]{CL26}. 
As a byproduct of the above UIT weak error estimate \eqref{weak}, we have the following optimal ergodic estimate between $\pi$ and $\pi_\tau^N$.

\begin{cor}\label{cor-erg} 
Let $X_0 \in \dot H^{1+\gamma}$ with $\gamma \in [0,1)$ satisfying \eqref{gamma} and Assumptions~\ref{ap-f}-\ref{ap-G} hold with $\lambda_1 > K_3 +\Lambda_{\rm w}$.
For any $\rho \in (0,1)$ and $\varphi \in \CC_b^3(H)$, there exist constants $C$ and $\tau_{\max} \in (0,1)$ such that for any $\tau \in (0, \tau_{\max}]$ and $N \in \nn_+$,
\begin{align*}
	\Big| \int_H \varphi(x) \pi(dx)- \int_{V_N}\varphi(x) \pi_\tau^N(dx)\Big|
\le C (\tau^\rho+\lambda_N^{-(\rho+\gamma/2)}).
\end{align*}
\end{cor}

\section{UIT A Priori Estimates}
\label{sec3}
 
This section provides the estimates required for the subsequent UIT weak error estimates in Section \ref{sec4}.
We first recall uniform $\dot H^1$- and $\dot H^{1+\gamma}$-moment bounds for the spectral Galerkin solution and prove the corresponding estimates for the GTEM and its interpolation.
Then, we establish regularity estimates for the BKE associated with \eqref{spe} and corresponding Malliavin derivative estimates.

\subsection{Moment and continuity estimates} 
Let us begin with the following UIT moment estimates of the spectral Galerkin approximation \eqref{spe} in the $\dot H^1$- and $\dot H^{1+\gamma}$-norm.
We omit the details of the proof, as it is analogous to that of \cite[Proposition 2.1]{Liu26} and \cite[Proposition 3.1]{LQ21} for \eqref{see}.

\begin{lm}  \label{lm-xn}
Let $p\ge2$, $X_0 \in \dot H^{1+\gamma}$ with $\gamma \in [0,1)$ satisfying \eqref{gamma}, and Assumption~\ref{ap-f} and \eqref{G1} hold. Assume $\lambda_1>K_3+\frac{p-1}{2}K_6$ for $\gamma=0$, and $\lambda_1 > K_3 + \frac{p(q+1)-1}{2} K_6$ for $\gamma \in (0,1)$. There exists a positive constant  $C$ such that 
\begin{equation}\label{xn}
  \ee \|X^N(t)\|_{1+\gamma}^{p} 
\le
\begin{cases}
C (1+\|X_0\|_1^{p}), & \gamma=0,\\
C (1+\|X_0\|_{1+\gamma}^{p(q+1)}), & \gamma \in (0,1).
\end{cases}
\end{equation}
\end{lm}

We need the following UIT estimates of \eqref{gtem} in the $\dot H^1$- and $\dot H^{1+\gamma}$-norms.

\begin{prop} \label{prop-h1+}
Let $p \in \nn_+$, $X_0 \in \dot H^{1+\gamma}$ with $\gamma \in [0,1)$ satisfying \eqref{gamma}, and Assumption~\ref{ap-f} and \eqref{G1} hold. Assume $\lambda_1> K_3+\frac{2p-1}{2}K_6$ for $\gamma=0$, and $\lambda_1> K_3+\frac{2p(q+1)-1}{2}K_6$ for $\gamma \in (0,1)$. There exist constants $C$ and $\tau_{\max} \in (0,1)$ such that for any $N \in \nn_+$ and $\tau \in (0,\tau_{\max}]$,
\begin{equation}\label{xnm-1+}
\sup_{m \in \nn_+}\ee\|X_m^N\|_{1+\gamma}^{2p}
\le
\begin{cases}
C(1+\|X_0\|_1^{2p}), & \gamma=0,\\
C(1+\|X_0\|_{1+\gamma}^{2p(q+1)}), & \gamma \in (0,1).
\end{cases}
\end{equation}
\end{prop}

\begin{proof}
To show \eqref{xnm-1+} with $\gamma=0$, we use \eqref{gtem} and the Poinc\'are inequality \eqref{poin} to get
\begin{align*}
(1+2\tau\lambda_1)\|X_{j+1}^N\|_1^2
 \le \|X_j^N+\tau \PP_N F_\tau(X_j^N) + \PP_N G(X_j^N)\delta_jW\|_1^2.
\end{align*}
For any $p \in \nn_+$, taking the $p$-th power and the conditional expectation $\ee_j[\cdot]:=\ee[\cdot \mid \FFF(t_j)]$ on both sides of the above inequality, we have
\begin{align*}
(1+2\tau\lambda_1)^p \ee_j \|X_{j+1}^N\|_1^{2p} 
\le \ee_j \|X_j^N+\tau \PP_N F_\tau(X_j^N) + \PP_N G(X_j^N)\delta_jW\|_1^{2p}.
\end{align*} 
Set
$Z_j(t):=X_j^N+\tau \PP_N F_\tau(X_j^N) + \PP_N G(X_j^N)(W(t_j+t)-W(t_j))$ for $0\le t\le \tau$.
Then $Z_j(0)=X_j^N+\tau \PP_N F_\tau(X_j^N)$ and $Z_j(\tau)=X_j^N+\tau \PP_N F_\tau(X_j^N) + \PP_N G(X_j^N)\delta_jW$. 

Applying It\^o formula to $\|Z_j(t)\|_1^{2p}$, followed by taking $\ee_j$, and utilizing the estimate $\|G(X_j^N)^*z\|_{U_0}^2\le \|G(X_j^N)\|_{\LL_2^1}^2\|z\|_1^2$, $z \in \dot H^1$, and H\"older inequality, we infer
\begin{align*}
\frac{\dd}{\dd t}\ee_j\|Z_j(t)\|_1^{2p}
\le p(2p-1)\|G(X_j^N)\|_{\LL_2^1}^2 \ee_j\|Z_j(t)\|_1^{2p-2}.
\end{align*} 
By H\"older inequality, we have $\ee_j\|Z_j(t)\|_1^{2p-2}\le (\ee_j\|Z_j(t)\|_1^{2p})^{\frac{p-1}{p}}$.
Substituting this into the above estimate yields $\frac{\dd}{\dd t}(\ee_j \|Z_j(t)\|_1^{2p})^{1/p} \le (2p-1)\|G(X_j^N)\|_{\LL_2^1}^2$.
Integrating over $[0,\tau]$ yields
\begin{align}\label{e_j}
& \ee_j\|X_j^N+\tau \PP_N F_\tau(X_j^N)+\PP_N G(X_j^N)\delta_jW\|_1^{2p} \nonumber \\
& \le (\|X_j^N+\tau F_\tau(X_j^N)\|_1^2+(2p-1)\tau \|G(X_j^N)\|_{\LL_2^1}^2 )^p.
\end{align}
This, in combination with \eqref{dis-mon} and \eqref{G1}, gives
\begin{align*}
\ee_j\|X_{j+1}^N\|_1^{2p}
\le \frac{1}{(1+2\tau\lambda_1)^p}\sum_{r=0}^p \mathcal B_r(\tau)\|X_j^N\|_1^{2r},
\end{align*}
where $\mathcal B_r(\tau):=\binom{p}{r}	(1+(2K_3+(2p-1)K_6)\tau+2K_3^2\tau^2)^r((2T_1+(2p-1)K_5)\tau)^{p-r}$,  $r=0, 1, \cdots, p$.
In particular, $\mathcal B_p(\tau)=(1+(2K_3+(2p-1)K_6)\tau+2K_3^2\tau^2)^p$.
When $p=1$, the sum $\sum_{r=1}^{p-1}$ is empty and adopted as zero by convention, allowing us to directly obtain the final iterative inequality without further bounding.  
For $p \ge 2$ and $r \in \{1,\cdots,p-1\}$, we apply Young inequality to obtain
\begin{align*}
\mathcal B_r(\tau)\|X_j^N\|_1^{2r}\le \frac{\Phi(\tau)}{2(p-1)}\|X_j^N\|_1^{2p}
+\frac{p-r}{p}(\frac{2r(p-1)}{p\Phi(\tau)})^{\frac{r}{p-r}}\mathcal B_r(\tau)^{\frac{p}{p-r}},
\end{align*}
where $\Phi(\tau):= (1+2\tau\lambda_1)^p - \mathcal B_p(\tau)$.
Summing over $r$ then yields
\begin{align*}
\sum_{r=1}^{p-1}\mathcal B_r(\tau)\|X_j^N\|_1^{2r}\le \frac{\Phi(\tau)}{2}\|X_j^N\|_1^{2p}
+\sum_{r=1}^{p-1}\frac{p-r}{p}(\frac{2r(p-1)}{p\Phi(\tau)})^{\frac{r}{p-r}}\mathcal B_r(\tau)^{\frac{p}{p-r}}.
\end{align*}
Combining the above estimates, we have
\begin{align}\label{ej-y}
\ee_j\|X_{j+1}^N\|_1^{2p}
\le (1-\frac{\Phi(\tau)}{2(1+2\tau\lambda_1)^p})\|X_j^N\|_1^{2p}+C(\tau),
\end{align}
with $C(\tau)=(\mathcal B_0(\tau) + \sum_{r=1}^{p-1}\frac{p-r}{p}(\frac{2r(p-1)}{p\Phi(\tau)})^{\frac{r}{p-r}}\mathcal B_r(\tau)^{\frac{p}{p-r}})/(1+2\tau\lambda_1)^p$, with the convention that the sum vanishes if $p=1$.

Using the elementary inequality $x^p-y^p\ge p(x-y)y^{p-1}$, $p \ge 1$, we obtain 
\begin{align}\label{Phi}
\Phi(\tau)\ge p(2\tau\lambda_1-(2K_3+(2p-1)K_6)\tau-2K_3^2\tau^2)\ge c_1\tau,
\end{align}
for some positive constant $c_1$ and for any $\tau \in (0, \tau_1)$, with $\tau_1:=\min\{1, (\lambda_1-K_3-\frac{2p-1}{2}K_6) K_3^{-2}\} \in (0, 1]$ as $\lambda_1>K_3+\frac{2p-1}{2}K_6$.
Moreover, $\mathcal B_0(\tau)=((2T_1+(2p-1)K_5))^p\tau^p$ and $\mathcal B_r(\tau)\le C\tau^{p-r}$, $1\le r\le p$.
Hence, for $1\le r\le p-1$ and $p\ge2$, $\Phi(\tau)^{-\frac r{p-r}}\mathcal B_r(\tau)^{\frac p{p-r}} \le C\tau^{p-\frac r{p-r}}\le C\tau$.
Thus $C(\tau)\le c_2\tau$ for some $c_2>0$, and \eqref{ej-y} and \eqref{Phi} yield \eqref{xnm-1+} with $\gamma=0$.

For $\gamma \in (0, 1)$, \eqref{full-sum} and the Minkovski inequality give
\begin{align*}
\|X_m^N\|_{L_\omega^{2p}\dot H^{1+\gamma}}
& \le \|S_{N, \tau}^m X^N_0\|_{L_\omega^{2p}\dot H^{1+\gamma}} 
+ \tau \sum_{i=0}^{m-1} \|S_{N, \tau}^{m-i} \PP_N F_\tau(X^N_i)\|_{L_\omega^{2p}\dot H^{1+\gamma}} \\
&\quad+ \|\sum_{i=0}^{m-1} S_{N, \tau}^{m-i} \PP_N G(X^N_i) \delta_i W \|_{L_\omega^{2p}\dot H^{1+\gamma}}.
\end{align*} 
The estimate \eqref{smo-pn} gives
$\|S_{N, \tau}^m X^N_0\|_{L_\omega^{2p}\dot H^{1+\gamma}}\le C\|X_0\|_{L_\omega^{2p}\dot H^{1+\gamma}}.$
Using \eqref{smo-pn}, \eqref{f-tau}, \eqref{f-grow}, \eqref{embedding}, and \eqref{xnm-1+} with $\gamma=0$, we get
\begin{align*}
\tau \sum_{i=0}^{m-1} \|S_{N, \tau}^{m-i} \PP_N F_\tau(X^N_i)\|_{L_\omega^{2p}\dot H^{1+\gamma}}
& \le C\tau \sum_{i=0}^{m-1} t_{m-i}^{-\frac{1+\gamma}{2}}e^{-ct_{m-i}}
\|F_\tau(X_i^N)\|_{L_\omega^{2p}L_\xi^2} \\
& \le C (1+\|X_0\|_{1+\gamma}^{q+1} ).
\end{align*}
By discrete Burkholder--Davis--Gundy (BDG) inequality, \eqref{smo-pn}, \eqref{G1}, and \eqref{xnm-1+} with $\gamma=0$, we infer that
\begin{align*}
\|\sum_{i=0}^{m-1} S_{N, \tau}^{m-i} \PP_N G(X^N_i) \delta_i W\|_{L_\omega^{2p}\dot H^{1+\gamma}}^2
& \le C \tau\sum_{i=0}^{m-1} t_{m-i}^{-\gamma}e^{-ct_{m-i}}
\|G(X_i^N)\|_{L_\omega^{2p}\LL_2^1}^2 \\
& \le C (1+\|X_0\|_{1+\gamma} )^2.
\end{align*} 
Combining the above three estimates, we conclude \eqref{xnm-1+}.
\end{proof}

Finally, we establish the following H\"older regularity estimate for the continuous interpolation \eqref{xhat} of the GTEM \eqref{gtem}.

\begin{lm}\label{lm-xhat}
Let $p \in \nn_+$, $X_0 \in \dot H^{1+\gamma}$ with $\gamma \in [0,1)$ satisfying \eqref{gamma}, and Assumption~\ref{ap-f} and \eqref{G1} hold. 
Assume $\lambda_1 > K_3+\frac{2p-1}{2} K_6$ for $\gamma=0$, and $\lambda_1 > K_3+\frac{2p(q+1)-1}{2} K_6$ for $\gamma \in (0,1)$.
Then there exist constants $C$ and $\tau_{\max} \in (0,1)$ such that for any $N \in \nn_+$ and $\tau \in (0,\tau_{\max}]$,
\begin{align}
&\sup_{m \in \nn} \ee \sup_{t \in [t_m,t_{m+1}]}\|\hat X^N(t)\|_{1+\gamma}^{2p}
 \le \begin{cases} 
 C(1+\|X_0\|_1^{2p}), & \gamma=0,\\
 C(1+\|X_0\|_{1+\gamma}^{2p(q+1)}), & \gamma \in (0,1).
\end{cases} \label{xhat-est}
\end{align}
Moreover, assume that either $\gamma=\nu=0$ and $\lambda_1>K_3+\frac{p(q+1)-1}{2}K_6$, or that $\gamma \in (0,1)$ satisfies \eqref{gamma}, $\nu \in [0,\gamma]$, and $\lambda_1>K_3+\frac{p(q+1)^2-1}{2}K_6$. Then
\begin{align}
\sup_{m \in \nn} \sup_{t_m \le s<t\le t_{m+1}}
\frac{\|\hat X^N(t)-\hat X^N(s)\|_{L_\omega^p\dot H^\nu}} {(t-s)^{1/2}}
\le \begin{cases} 
C(1+\|X_0\|_1^{q+1}), & \gamma=0,\\
C (1+\|X_0\|_{1+\gamma}^{(q+1)^2}), & \gamma \in (0,1).
\end{cases}\label{holder}
\end{align}
\end{lm}

\begin{proof}
Let $m \in \nn$.
For any $t \in [t_m,t_{m+1}]$, by \eqref{xhat}, we get
\begin{align*}
\|\hat X^N(t)\|_{1+\gamma}
& \le C \|X^N_m\|_{1+\gamma} + \tau\|S_{N,\tau}\PP_NF_\tau(X_m^N)\|_{1+\gamma} \\
&\quad+ \Big\| \int_{t_m}^t S_{N,\tau}\PP_NG(X^N_m)\dd W(s)\Big\|_{1+\gamma}.
\end{align*}
After taking the supremum over $t \in [t_m, t_{m+1}]$ and taking the $L^{2p}_\omega$-norm, we apply BDG inequality, \eqref{f-grow}, \eqref{f-tau}, \eqref{embedding}, and \eqref{G1} to obtain
\begin{align*}
& \ee  \sup_{t \in [t_m, t_{m+1}]} \|\hat X^N(t)\|_{1+\gamma}^{2p} 
\le C \ee\|X^N_m\|_{1+\gamma}^{2p} + C \tau^{(1-\gamma)p} ( 1 + \ee\|X^N_m\|_1^{2p(q+1)} ).
\end{align*} 
Hence, the first estimate \eqref{xhat-est} with $\gamma \in (0,1)$ follows from \eqref{xnm-1+}. 
Similarly, by \eqref{xhat}, BDG inequality, \eqref{tauFtau1}, and \eqref{G1},
\begin{align*}
\ee\sup_{t \in [t_m,t_{m+1}]}\|\hat X^N(t)\|_1^{2p}
&\le C\ee\|X_m^N\|_1^{2p}
+C\tau^p(1+\ee\|X_m^N\|_1^{2p}),
\end{align*}
which, together with \eqref{xnm-1+} with $\gamma=0$,
proves \eqref{xhat-est} with $\gamma=0$.
Next, by \eqref{xhat}, Minkowski inequality, \eqref{smo-pn}, \eqref{F-bounds-0}, the BDG inequality, and \eqref{G1}, we obtain, for $t_m\le s<t\le t_{m+1}$,
\begin{align*}
& \|\hat X^N(t)-\hat X^N(s)\|_{L_\omega^p\dot H^\nu} \\
& \le \|(t-s)A_NS_{N,\tau}X_m^N\|_{L_\omega^p\dot H^\nu} +\|(t-s)S_{N,\tau}\PP_NF_\tau(X_m^N)\|_{L_\omega^p\dot H^\nu} \\
&\quad +\|S_{N,\tau}\PP_NG(X_m^N)(W(t)-W(s))\|_{L_\omega^p\dot H^\nu} \\
& \le C(t-s)\tau^{-\frac{1+\nu-\gamma}{2}}\|X_m^N\|_{L_\omega^p\dot H^{1+\gamma}} +C(t-s)(1+\|X_m^N\|_{L_\omega^{p(q+1)}\dot H^{1+\gamma}}^{q+1}) \\
&\quad +C(t-s)^{\frac12}(1+\|X_m^N\|_{L_\omega^p\dot H^1}) \\
& \le C(t-s)^{\frac12} (1+\|X_m^N\|_{L_\omega^{p(q+1)}\dot H^{1+\gamma}}^{q+1}).
\end{align*}
Applying \eqref{xnm-1+} with moment order $p(q+1)$, we conclude \eqref{holder}.

\end{proof}

\subsection{Regularity of BKE}

In this part, we aim to establish UIT regularity results for the BKE associated with Eq. \eqref{spe}. 

For any test function $\varphi \in \CC_b^3(H)$ and $M \gg N$, we define
\begin{align} \label{def-u}
    U^M(t,x) := \ee \varphi(X^M(t,x)), \quad t \ge 0, \ x \in V_M.
\end{align}
It\^o formula yields that $U^M$ satisfies the following BKE:
\begin{align} \label{kol}
    \partial_t U^M(t,x) &= \< A_M x + \PP_M F(x), D U^M(t,x) \> \nonumber \\
    &\quad + \frac12 \operatorname{Tr} [ D^2U^M(t,x) (\PP_M G(x) Q^\frac12)(\PP_M G(x) Q^\frac12)^* ], 
\end{align}
$(t, x) \in \rr_+ \times V_M$, with $U^M(0,\cdot) = \varphi(\cdot)$. 
For notational convenience, sometimes we write $U_t^M(\cdot):=U^M(t,\cdot)$ for $t \in \rr_+$ whenever no confusion arises.

For the estimates of higher-order derivatives of the solution to the BKE \eqref{kol}, we fix $1\le p_2 < p_\Theta<p_1 < 2p <p^*$ with  $R:=2p_2(8q-2)=2p_\Theta(7q-2)=2p_1(4q-1)=4pq$.
Accordingly, for $t > 0$, $x \in V_M$, and $j \in \{1,2,\Theta,3\}$, define
\begin{align*}
(r_1,r_2,r_\Theta,r_3):=(q,4q-1,7q-2,8q-2),\quad 
\Psi_t^{(j)}(x):=1+\sup_{0\le s\le t} \|X^M(s,x)\|_{L_\omega^R L_\xi^ \infty}^{r_j}.
\end{align*}
For later use, set $\kappa_{d,\varrho}:=\max\{\frac d4-\varrho,0\}$, $\kappa_{d,\chi}:=\max\{\frac d2-1-\frac{\chi}{2},0\}$, and $\kappa_{d,\varrho,\chi}:=\max\{\kappa_{d,\varrho},\kappa_{d,\chi}\}$.

\begin{lm}\label{lm-DX}
Let $q\ge2$, $1\le p_2<p_\Theta<p_1<2p<p^*$, and Assumptions \ref{ap-f}--\ref{ap-G} hold with $\lambda_1>K_3+\frac{4p-1}{2}K_4^2$.
Let $\varrho \in (1/4,1/2)$, $\gamma \in [0,1)$ satisfy \eqref{gamma}, and let $\chi \in [0,\gamma]$  satisfy $\chi+2\varrho>d/2$. 
For $\alpha,\vartheta,\vartheta_1,\vartheta_2,\varepsilon_1,\varepsilon_2 \in [0,1/2)$ with $\vartheta_1+\vartheta_2+\kappa_{d,\varrho}<\frac12$ and $\varepsilon_1+\varepsilon_2 <\min\{\frac12-\kappa_{d,\varrho,\chi},\varrho-\kappa_{d,\chi}\}$, and for any sufficiently small $\delta>0$, there exist constants $C>0$ and $c>0$ such that for all $x,h,k,l \in V_M$, and $t>0$,
\begin{align}
\|D X^M(t,x)h\|_{L_\omega^{4p}} &\le C(1+t^{-\alpha})e^{-ct}\Psi_t^{(1)}(x)\|h\|_{-2\alpha}, \label{DX-est}\\
\|D^2X^M(t,x)(h,k)\|_{L_\omega^{2p_1}} &\le C(1+t^{-\vartheta_1-\vartheta_2-\kappa_{d,\varrho}-\delta}) e^{-ct}\Psi_t^{(2)}(x) \|h\|_{-2\vartheta_1}\|k\|_{-2\vartheta_2},\label{D2X-est}\\
\|D^2X^M(t,x)(h,k)\|_{L_\omega^{2p_1}} &\le C(1+t^{-\vartheta})e^{-ct}\Psi_t^{(2)}(x) \|h\|_{-2\vartheta}\|k\|_{\chi}, \label{D2X-mix-est}\\
\|D^3X^M(t,x)(h,k,l)\|_{L_\omega^{2p_2}} &\le C(1+t^{-\varepsilon_1-\varepsilon_2-\kappa_{d,\varrho,\chi}-\delta}) e^{-ct}\Psi_t^{(3)}(x) \|h\|_{-2\varepsilon_1}\|k\|_{-2\varepsilon_2}\|l\|_{\chi}. \label{D3X-mix-est}
\end{align}
\end{lm}

\begin{proof}

Let $t>0$ and $x, h,k,l \in V_M$. 
In what follows, we denote
\begin{align*}
    \eta^{h}(t, x) := D X^M(t,x) h, 
   &  \quad \widetilde{\eta}^h(t,x):=\eta^h(t,x)-e^{tA}h, \\
    \zeta^{h,k}(t, x) := D^2 X^M(t, x) (h, k), 
    & \quad \psi^{h,k,l}(t, x) := D^3 X^M(t, x) (h, k,l).
\end{align*} 
For notational brevity, we suppress the $x$-dependence of $X^M$, $\eta$, $\zeta$, and $\psi$ whenever no confusion arises. 
For brevity, define $\LL_s y:=A_My+\PP_MDF(X^M(s))y$, $a_h(s):=\PP_MDF(X^M(s))e^{sA}h$, and $b_h(s):=\PP_MDG(X^M(s))e^{sA}h$ for $0<s\le t$.
Then for the terms $\eta^{h}$ and $\widetilde{\eta}^h$, we have $\eta^h(0)=h$, $\widetilde{\eta}^h(0)=0$, and for $0<s\le t$, 
\begin{align*}
\dd \widetilde{\eta}^h(s)
= [\LL_s\widetilde{\eta}^h(s)+a_h(s)]\dd s +[\PP_MDG(X^M(s))\widetilde{\eta}^h(s)+b_h(s)]\dd W(s), \quad \widetilde{\eta}^h(0)=0.
\end{align*}
Applying It\^o formula to $\|\widetilde{\eta}^h(s)\|^{4p}$, taking expectations, and using Young inequality, we obtain
\begin{align*}
\frac{1}{4p}\frac{\dd}{\dd s}\ee\|\widetilde{\eta}^h(s)\|^{4p} 
&\le \ee[\|\widetilde{\eta}^h(s)\|^{4p-2}\< \LL_s\widetilde{\eta}^h(s),\widetilde{\eta}^h(s)\>] +\ee[\|\widetilde{\eta}^h(s)\|^{4p-1}\|a_h(s)\|]  \\
&\quad +\frac{4p-1}{2}(1+\varepsilon) \ee[\|\widetilde{\eta}^h(s)\|^{4p-2}\|\PP_MDG(X^M(s))\widetilde{\eta}^h(s)\|_{\LL_2^0}^2] \\
&\quad +C_\varepsilon\ee[\|\widetilde{\eta}^h(s)\|^{4p-2}\|b_h(s)\|_{\LL_2^0}^2].
\end{align*}
Choosing $\varepsilon>0$ sufficiently small and using \eqref{D-coup}, we obtain
\begin{align*}
\frac{\dd}{\dd s}\ee\|\widetilde{\eta}^h(s)\|^{4p}
&\le -K \ee\|\widetilde{\eta}^h(s)\|^{4p} +C\ee[\|\widetilde{\eta}^h(s)\|^{4p-1}\|a_h(s)\|]  \\
&\quad +C\ee[\|\widetilde{\eta}^h(s)\|^{4p-2} \|b_h(s)\|_{\LL_2^0}^2].
\end{align*}
Setting $z_1(s):=\ee\|\widetilde{\eta}^h(s)\|^{4p}$, H\"older inequality gives
\begin{align*}
z_1'(s)
\le -K z_1(s) +Cz_1(s)^{1-\frac{1}{4p}}\|a_h(s)\|_{L_\omega^{4p}} +Cz_1(s)^{1-\frac{1}{2p}}\|b_h(s)\|_{L_\omega^{4p}\LL_2^0}^2.
\end{align*}
For $\iota>0$, define $R_\iota(s):=(z_1(s)+\iota)^{1/(2p)}$.  
Using
\begin{align*}
(z_1+\iota)^{\frac1{2p}-1}z_1 \ge R_\iota-\iota^{1/(2p)},~
(z_1+\iota)^{\frac1{2p}-1}z_1^{1-\frac1{4p}} \le R_\iota^{1/2},~
(z_1+\iota)^{\frac1{2p}-1}z_1^{1-\frac1{2p}} \le 1,
\end{align*}
and Young inequality, we obtain for some $c>0$,
\begin{align*}
R_\iota'(s)
\le -cR_\iota(s) +C(\|a_h(s)\|_{L_\omega^{4p}}^2 +\|b_h(s)\|_{L_\omega^{4p}\LL_2^0}^2 +\iota^{1/(2p)}).
\end{align*}
Gronwall inequality and the limit $\iota\downarrow0$ yield
\begin{align*}
\|\widetilde{\eta}^h(t)\|_{L_\omega^{4p}}^2
\le C \int_0^t e^{-c(t-s)} (\|a_h(s)\|_{L_\omega^{4p}}^2 +\|b_h(s)\|_{L_\omega^{4p}\LL_2^0}^2)\dd s.
\end{align*}
By \eqref{F-bounds-1}, \eqref{DG}, and \eqref{smo},
\begin{align*}
\|a_h(s)\|_{L_\omega^{4p}}^2 +\|b_h(s)\|_{L_\omega^{4p}\LL_2^0}^2
\le C|\Psi_t^{(1)}(x)|^2s^{-2\alpha}e^{-2\lambda_1s} \|h\|_{-2\alpha}^2.
\end{align*}
Since $\alpha<1/2$, the last two estimates imply
\begin{align*}
\|\widetilde{\eta}^h(t)\|_{L_\omega^{4p}}^2 
&\le Ce^{-ct}|\Psi_t^{(1)}(x)|^2\|h\|_{-2\alpha}^2.
\end{align*}
Together with $\eta^h(t,x)=\widetilde\eta^h(t,x)+e^{tA}h$ and \eqref{smo}, this proves \eqref{DX-est}. 

Let $\beta,\sigma\ge0$ satisfy $\beta+\sigma/2<1/2$. From the mild form
\begin{align*}
\eta^h(s,x)
&=S_M(s)h+ \int_0^s S_M(s-r)\PP_MDF(X^M(r))\eta^h(r)\dd r\\
&\quad+ \int_0^s S_M(s-r)\PP_MDG(X^M(r))\eta^h(r)\dd W(r),
\end{align*}
\eqref{smo}, \eqref{F-bounds-1}, \eqref{DG}, H\"older and BDG inequalities, and \eqref{DX-est}, we obtain  
\begin{align}
\|\eta^h(s,x)\|_{L_\omega^{2p}\dot H^\sigma}
\le C(1+s^{-\beta-\sigma/2})e^{-cs} |\Psi_t^{(1)}(x)|^2\|h\|_{-2\beta},
\quad 0<s\le t.
\label{DX-spatial}
\end{align}
The same mild-form argument also gives, for $0\le\sigma<1$ and $h \in \dot H^\sigma\cap V_M$, 
\begin{align}
\|\eta^h(s,x)\|_{L_\omega^{2p}\dot H^\sigma}
\le C e^{-cs}|\Psi_t^{(1)}(x)|^2\|h\|_\sigma,
\quad 0<s\le t.
\label{DX-pos}
\end{align}

We next turn to the second variation. We have $\zeta^{h,k}(0)=0$ and, for $0<s\le t$,
\begin{align*}
\dd \zeta^{h,k}(s)
&=[\LL_s\zeta^{h,k}(s)+\PP_Ma(s)]\dd s\\
&\quad+[\PP_MDG(X^M(s))\zeta^{h,k}(s)+\PP_Mb(s)]\dd W(s),
\end{align*}
where $a(s):=D^2F(X^M(s))(\eta^h(s),\eta^k(s))$ and $b(s):=D^2G(X^M(s))(\eta^h(s),\eta^k(s))$.  

By It\^o formula, the duality between $\dot H^{-2\varrho}$ and $\dot H^{2\varrho}$, the embedding $\dot H^1\hookrightarrow\dot H^{2\varrho}$, Young inequality, and \eqref{D-coup}, we obtain
\begin{align*}
\frac{\dd}{\dd s}\ee\|\zeta^{h,k}(s)\|^{2p_1}
\le -c\ee\|\zeta^{h,k}(s)\|^{2p_1} +C\ee\Big[\|\zeta^{h,k}(s)\|^{2p_1-2} (\|a(s)\|_{-2\varrho}^2+\|b(s)\|_{\LL_2^0}^2)\Big].
\end{align*}
Thus, with $z_2(s):=\ee\|\zeta^{h,k}(s)\|^{2p_1}$, H\"older inequality gives
\begin{align*}
z_2'(s) \le -cz_2(s) +Cz_2(s)^{1-\frac1{p_1}}\Big( \|a(s)\|_{L_\omega^{2p_1}\dot H^{-2\varrho}}^2 +\|b(s)\|_{L_\omega^{2p_1}\LL_2^0}^2\Big).
\end{align*}
Arguing as in the proof of \eqref{DX-est}, namely applying the $R_\iota$-regularization to $(z_2+\iota)^{1/p_1}$, using $\zeta^{h,k}(0)=0$, Gronwall inequality, and then letting $\iota\downarrow0$, we obtain
\begin{align}
\|\zeta^{h,k}(t)\|_{L_\omega^{2p_1}}^2
&\le C \int_0^t e^{-c(t-s)} ( \|a(s)\|_{L_\omega^{2p_1}\dot H^{-2\varrho}}^2 +\|b(s)\|_{L_\omega^{2p_1}\LL_2^0}^2 )\dd s.
\label{zeta}
\end{align}
It remains to estimate $a(s)$ and $b(s)$. Set $\mu_1=0$ and $\mu_2=2\kappa_{d,\varrho}+2\delta$. For sufficiently small $\delta>0$, we have $\mu_2 \in [0,\min\{1,d/2\})$, $\mu_1+\mu_2+2\varrho>d/2$, and $\vartheta_i+\mu_i/2<1/2$, $i=1,2$. 
Hence, by \eqref{D2F-d}, H\"older inequality, \eqref{DX-est}, and \eqref{DX-spatial},
\begin{align*}
\|a(s)\|_{L_\omega^{2p_1}\dot H^{-2\varrho}} 
&\le C_\delta(1+s^{-\vartheta_1-\vartheta_2-\kappa_{d,\varrho}-\delta})e^{-cs} \Psi_t^{(2)}(x) \|h\|_{-2\vartheta_1}\|k\|_{-2\vartheta_2}.
\end{align*}
Moreover, by \eqref{D2G} and \eqref{DX-est},
\begin{align*}
\|b(s)\|_{L_\omega^{2p_1}\LL_2^0}
&\le C(1+s^{-\vartheta_1-\vartheta_2})e^{-cs}\Psi_t^{(2)}(x)\|h\|_{-2\vartheta_1}\|k\|_{-2\vartheta_2}.
\end{align*}
Substituting these two estimates into \eqref{zeta}, and using $\vartheta_1+\vartheta_2+\kappa_{d,\varrho}+\delta<1/2$, gives
\begin{align*}
\|\zeta^{h,k}(t)\|_{L_\omega^{2p_1}}
\le C(1+t^{-\vartheta_1-\vartheta_2-\kappa_{d,\varrho}-\delta}) e^{-ct}\Psi_t^{(2)}(x) \|h\|_{-2\vartheta_1}\|k\|_{-2\vartheta_2},
\end{align*}
which proves \eqref{D2X-est}.

For \eqref{D2X-mix-est}, let $h \in V_M$ and $k \in \dot H^\chi\cap V_M$. Since $\chi+2\varrho>d/2$, by \eqref{D2F-d} with $(\mu_1,\mu_2)=(0,\chi)$, \eqref{D2G}, H\"older inequality, \eqref{DX-est} with $\alpha=\vartheta$, and \eqref{DX-pos}, we obtain
\begin{align*}
	\|a(s)\|_{L_\omega^{2p_1}\dot H^{-2\varrho}} +\|b(s)\|_{L_\omega^{2p_1}\LL_2^0}
\le C(1+s^{-\vartheta})e^{-cs}\Psi_t^{(2)}(x) \|h\|_{-2\vartheta}\|k\|_\chi .
\end{align*}
Together with \eqref{zeta}, this gives \eqref{D2X-mix-est}.

For the third variation, separate the $D^3F$-term by setting
\begin{align*}
\Theta^{h,k,l}(t) := \int_0^t S_M(t-s)\PP_MD^3F(X^M(s)) (\eta^{h}(s),\eta^{k}(s),\eta^{l}(s))\dd s.
\end{align*} 
Choose $\nu_1,\nu_2 \in [0,\min\{1,d/2\})$ such that
\begin{align*}
\nu_1+\nu_2+\chi+\min\{2\varrho,\frac d 2\}>d,\quad
\varepsilon_i+\frac{\nu_i}{2}<\frac12,\ i=1,2,\quad
\frac{\nu_1+\nu_2}{2}< \kappa_{d,\chi}+1-\varrho+\delta .
\end{align*}
This is possible since $\varepsilon_1+\varepsilon_2+\kappa_{d,\chi}+\delta<\varrho$.
By \eqref{D3F-d} with $(\nu_1,\nu_2, \nu_3)=(\nu_1,\nu_2,\chi)$, H\"older inequality, \eqref{DX-spatial} with
$(\beta,\sigma)=(\varepsilon_i,\nu_i)$ for $i=1,2$, and \eqref{DX-pos} for the third direction,
\begin{align*}
&\|D^3F(X^M(s))(\eta^{h}(s),\eta^{k}(s),\eta^{l}(s))
\|_{L_\omega^{2p_\Theta}\dot H^{-2\varrho}}\\
&\quad\le C (1+s^{-\varepsilon_1-\nu_1/2}) (1+s^{-\varepsilon_2-\nu_2/2}) e^{-cs}\Psi_t^{(\Theta)}(x) \|h\|_{-2\varepsilon_1} \|k\|_{-2\varepsilon_2} \|l\|_\chi.
\end{align*}
Hence, by \eqref{smo},
\begin{align}
\|\Theta^{h,k,l}(t)\|_{L_\omega^{2p_\Theta}}
&\le C_\delta \int_0^t(t-s)^{-\varrho}e^{-c(t-s)}(1+s^{-\varepsilon_1-\varepsilon_2-\kappa_{d,\chi}-1+\varrho-\delta}) e^{-cs}\dd s \notag\\
&\quad\times \Psi_t^{(\Theta)}(x) \|h\|_{-2\varepsilon_1} \|k\|_{-2\varepsilon_2} \|l\|_\chi \notag\\
&\le C_\delta (1+t^{-\varepsilon_1-\varepsilon_2-\kappa_{d,\chi}-\delta}) e^{-ct}\Psi_t^{(\Theta)}(x) \|h\|_{-2\varepsilon_1} \|k\|_{-2\varepsilon_2} \|l\|_\chi.
\label{Theta-mix}
\end{align}

Define $\bar\psi^{h,k,l}(t):=\psi^{h,k,l}(t)-\Theta^{h,k,l}(t)$.
Then $\bar\psi^{h,k,l}(0)=0$ and
\begin{align*}
\dd\bar\psi^{h,k,l}(s)
&=[\LL_s\bar\psi^{h,k,l}(s)+\PP_M\bar H(s)]\dd s\\
&\quad+[\PP_MDG(X^M(s))\bar\psi^{h,k,l}(s)+\PP_M\bar J(s)]\dd W(s),
\end{align*}
where
\begin{align*}
\bar H(s)
&:=D^2F(X^M(s))(\zeta^{k,l}(s),\eta^h(s))+D^2F(X^M(s))(\zeta^{h,k}(s),\eta^l(s))\\
&\quad+D^2F(X^M(s))(\zeta^{h,l}(s),\eta^k(s))+DF(X^M(s))\Theta^{h,k,l}(s),\\
\bar J(s)
&:=D^2G(X^M(s))(\zeta^{k,l}(s),\eta^h(s))+D^2G(X^M(s))(\zeta^{h,k}(s),\eta^l(s))\\
&\quad+D^2G(X^M(s))(\zeta^{h,l}(s),\eta^k(s))+D^3G(X^M(s))(\eta^h(s),\eta^k(s),\eta^l(s))\\
&\quad+DG(X^M(s))\Theta^{h,k,l}(s).
\end{align*}
For the $D^2F$-terms in $\bar H$, we distinguish whether the $\dot H^\chi$-direction is carried by $\eta$ or by $\zeta$. If it is carried by $\eta^l$, then \eqref{D2F-d} with $(\mu_1,\mu_2)=(0,\chi)$, H\"older inequality, \eqref{D2X-est}, and \eqref{DX-pos} give
\begin{align*}
&\|D^2F(X^M(s))(\zeta^{h,k}(s),\eta^l(s)) \|_{L_\omega^{2p_2}\dot H^{-2\varrho}}\\
&\quad\le C_\delta(1+s^{-\varepsilon_1-\varepsilon_2-\kappa_{d,\varrho}-\delta})e^{-cs} \Psi_t^{(3)}(x) \|h\|_{-2\varepsilon_1} \|k\|_{-2\varepsilon_2} \|l\|_\chi .
\end{align*}
If it is carried by
$\zeta^{k,l}$, choose $\mu\in[0,\min\{1,d/2\})$ such that
$
\mu+2\varrho>\frac d2,
\frac{\mu}{2}<\kappa_{d,\varrho}+\delta,
\varepsilon_i+\frac{\mu}{2}<\frac12, i=1,2$.
Applying \eqref{D2F-d} with $(\mu_1,\mu_2)=(0,\mu)$,
H\"older inequality, \eqref{D2X-mix-est} with
$\vartheta=\varepsilon_2$, and \eqref{DX-spatial} with
$(\beta,\sigma)=(\varepsilon_1,\mu)$, and estimating the
term involving $\zeta^{h,l}$ analogously by interchanging
$(h,\varepsilon_1)$ and $(k,\varepsilon_2)$, we obtain
\begin{align*}
&\|D^2F(X^M(s))(\zeta^{k,l}(s),\eta^h(s))\|_{L_\omega^{2p_2}\dot H^{-2\varrho}}+\|D^2F(X^M(s))(\zeta^{h,l}(s),\eta^k(s)) \|_{L_\omega^{2p_2}\dot H^{-2\varrho}} \\
&\quad\le C (1+s^{-\varepsilon_1-\varepsilon_2-\kappa_{d,\varrho}-\delta})e^{-cs} \Psi_t^{(3)}(x) \|h\|_{-2\varepsilon_1} \|k\|_{-2\varepsilon_2} \|l\|_\chi .
\end{align*} 
Moreover, by \eqref{F-bounds-1}, the embedding $H\hookrightarrow\dot H^{-2\varrho}$, and \eqref{Theta-mix},
\begin{align*}
&\|DF(X^M(s))\Theta^{h,k,l}(s)\|_{L_\omega^{2p_2}\dot H^{-2\varrho}}\\
\le&~ C (1+s^{-\varepsilon_1-\varepsilon_2-\kappa_{d,\chi}-\delta})e^{-cs} \Psi_t^{(3)}(x) \|h\|_{-2\varepsilon_1} \|k\|_{-2\varepsilon_2} \|l\|_\chi .
\end{align*}
Together with the estimates for the $D^2F$-terms, this gives
\begin{align*}
\|\bar H(s)\|_{L_\omega^{2p_2}\dot H^{-2\varrho}}
&\le C (1+s^{-\varepsilon_1-\varepsilon_2-\kappa_{d,\varrho,\chi}-\delta})e^{-cs} \Psi_t^{(3)}(x) \|h\|_{-2\varepsilon_1} \|k\|_{-2\varepsilon_2} \|l\|_\chi .
\end{align*}
Similarly, by \eqref{D2G}, \eqref{D3G}, \eqref{DG}, \eqref{D2X-est},
\eqref{D2X-mix-est}, \eqref{DX-est}, and \eqref{Theta-mix},
\begin{align*}
\|\bar J(s)\|_{L_\omega^{2p_2}\LL_2^0}
&\le C (1+s^{-\varepsilon_1-\varepsilon_2-\kappa_{d,\varrho,\chi}-\delta})e^{-cs} \Psi_t^{(3)}(x) \|h\|_{-2\varepsilon_1} \|k\|_{-2\varepsilon_2} \|l\|_\chi .
\end{align*}
Applying the same energy argument as for $\zeta^{h,k}$ to $\bar\psi^{h,k,l}$ and using $\varepsilon_1+\varepsilon_2+\kappa_{d,\varrho,\chi}+\delta<1/2$, we obtain
\begin{align}
\|\bar\psi^{h,k,l}(t)\|_{L_\omega^{2p_2}}
&\le C(1+t^{-\varepsilon_1-\varepsilon_2-\kappa_{d,\varrho,\chi}-\delta}) e^{-ct}\Psi_t^{(3)}(x) \|h\|_{-2\varepsilon_1}\|k\|_{-2\varepsilon_2}\|l\|_\chi .
\label{psibar-mix}
\end{align}
Since $\psi^{h,k,l}=\bar\psi^{h,k,l}+\Theta^{h,k,l}$, \eqref{Theta-mix} and \eqref{psibar-mix} imply \eqref{D3X-mix-est}. 
\end{proof}

By the chain rule and the definition \eqref{def-u}, we have the following UIT regularity results for the solution to the BKE \eqref{kol}.

\begin{cor}\label{cor-DU}
Under the conditions in Lemma \ref{lm-DX}, for any $\varphi \in \CC_b^3(H)$, there exist positive constants $C$ and $c$ such that 
\begin{align}
|D U^M(t, x) h| & \leq C (1+t^{-\alpha}) e^{- ct} \Psi_t^{(1)}(x) \|h\|_{-2\alpha}, \label{Du}\\
|D^2 U^M(t, x) (h, k) | & \leq C (1+t^{-\vartheta_1-\vartheta_2-\kappa_{d,\varrho}-\delta} ) e^{- ct} \Psi_t^{(2)}(x) \|h\|_{-2\vartheta_1}\|k\|_{-2\vartheta_2},  \label{D2u} \\
|D^2 U^M(t, x) (h, k) | & \leq C (1+t^{-\vartheta} ) e^{- ct} \Psi_t^{(2)}(x) \|h\|_{-2\vartheta}\|k\|_{\chi},  \label{D2u-mix} \\
|D^3 U^M(t, x) (h,k,l) | & \leq  C (1+t^{-\varepsilon_1-\varepsilon_2-\kappa_{d,\varrho,\chi}-\delta} ) e^{- ct} \Psi_t^{(3)}(x) \|h\|_{-2\varepsilon_1}\|k\|_{-2\varepsilon_2}\|l\|_{\chi},  \label{D3u}
\end{align}   
where the parameters are as in Lemma \ref{lm-DX}.
\end{cor}

\begin{rk}
We point out that, in the case of additive noise, Lemma \ref{lm-DX} and Corollary \ref{cor-DU} remain valid for parameters $\alpha, \vartheta_1, \vartheta_2, \varepsilon_1, \varepsilon_2, \varepsilon_3 \in [0,1)$ satisfying $\alpha < 1/2$, $\vartheta_1+\vartheta_2 < 1/2$, and $\varepsilon_1+\varepsilon_2+\varepsilon_3<1/2$; see, e.g., \cite[Lemma 5]{CHS21} and \cite[Lemma~5.6]{JW25}. 
\end{rk}

\subsection{Malliavin regularity estimate of GTEM}
To control the stochastic terms appearing in the weak error decomposition in Section~\ref{sec4}, we record Malliavin estimates for the GTEM \eqref{gtem} and its interpolation \eqref{xhat}.

Let $\{g_l\}_{l\ge1}$ be the eigenbasis of $Q$ introduced in Section~\ref{sec2}, and set $\widetilde g_l:=Q^{1/2}g_l$, $l \in \nn_+$. 
For a $V_N$-valued random variable $Y$, its Malliavin derivative is the process $r\mapsto \DD_r Y \in \LL_2(U_0;V_N)$; we write $\DD_r^lY:=(\DD_rY) \widetilde g_l$, $l \in \nn_+$.
The following Malliavin chain rule and IBP formula hold for any $\varphi \in C^1(V_N)$ and any adapted $\Gamma \in L^2(\Omega\times[0, T];\LL_2^0)$:  
\begin{align} 
\DD_r[\varphi(Y)]& =D\varphi(Y)\DD_rY, \quad Y \in \mathbb D^{1,p}(V_N), \label{chain} \\ 
\ee\Big\langle Z, \int_0^ \infty\Gamma(r)\dd W(r)\Big\rangle
& =\ee \int_0^ \infty\langle \DD_rZ,\Gamma(r)\rangle_{\LL_2^0}\dd r,
\quad Z \in \mathbb D^{1,2}(V_N). \label{ibp}
\end{align}

\begin{lm}\label{lm-mall-p}
Let $p \in \nn_+$, $X_0 \in \dot H^1$, and Assumptions \ref{ap-f}-\ref{ap-G} hold with $\lambda_1>K_3+\frac{2p-1}{2}\max\{K_4^2,K_6\}$.
Then there exist $C,c>0$ and $\tau_{\max} \in (0,1)$ such that for any $k\ge1$, $0\le s\le t_k$, $t \in [t_k,t_{k+1})$, $\tau \in (0,\tau_{\max}]$, and $0\le r\le t$,
\begin{align}
&\|\DD_sX_k^N\|_{L_\omega^{2p}\LL_2^0} \le Ce^{-c(t_k-[s]_\tau)}(1+\|X_0\|_1), \label{DM-pointwise-p}\\
&\|\DD_r\hat X^N(t)\|_{L_\omega^{2p}\LL_2^0} \le Ce^{-c(t-[r]_\tau)}(1+\|X_0\|_1). \label{mall-est-local-p}
\end{align}
\end{lm}

\begin{proof}
Fix $s \in [t_m,t_{m+1})$, differentiating \eqref{full} gives
\begin{align}\label{DM-p}
\DD_s^lX_j^N=0,\quad 0\le j\le m;\quad
\DD_s^lX_{m+1}^N=S_{N,\tau}\PP_NG(X_m^N)\widetilde g_l,
\end{align}
and for $j\ge m+1$, differentiating \eqref{gtem} yields 
\begin{align}  \label{DM-rec-p}
\DD_sX_{j+1}^N
& = \DD_sX_j^N+\tau A_N\DD_sX_{j+1}^N +\tau\PP_NDF_\tau(X_j^N)\DD_sX_j^N \nonumber \\
& \quad +\PP_NDG(X_j^N)(\DD_sX_j^N)\delta_j W,
\end{align}
where the last term is understood columnwise, namely,
\begin{align*}
 (\PP_NDG(X_j^N)(\DD_sX_j^N)\delta_j W )u
:=\PP_NDG(X_j^N)(\DD_sX_j^Nu)\delta_j W,
\quad u \in U_0.
\end{align*} 
Set $Y_j:=\DD_sX_j^N \in \LL_2^0$. By \eqref{DM-rec-p} and the Poincar\'e inequality \eqref{poin}, we have
\begin{align*}
(1+2\lambda_1\tau)\|Y_{j+1}\|_{\LL_2^0}^2
\le \|Y_j+\tau\PP_NDF_\tau(X_j^N)Y_j+\PP_NDG(X_j^N)(Y_j)\delta_j W\|_{\LL_2^0}^2.
\end{align*}
Taking the $p$-th power and the conditional expectation $\ee_j$ on both sides yields
\begin{align*}
(1+2\tau\lambda_1)^p \ee_j \|Y_{j+1}\|_{\LL_2^0}^{2p}
\le \ee_j \|Y_j+\tau\PP_NDF_\tau(X_j^N)Y_j+\PP_NDG(X_j^N)(Y_j)\delta_jW\|_{\LL_2^0}^{2p}.
\end{align*} 
By an argument similar to \eqref{e_j}, combining \eqref{ftau'+} with \eqref{DG} gives
\begin{align}\label{mall-contract-p}
\ee_j\|Y_{j+1}\|_{\LL_2^0}^{2p}
&\le \Big(\frac{1+(2K_3+(2p-1)K_4^2)\tau+2K_3^2\tau^2}{1+2\tau\lambda_1}\Big)^p \|Y_j\|_{\LL_2^0}^{2p} 
\le (1-c\tau)\|Y_j\|_{\LL_2^0}^{2p},
\end{align}
where $c:=\frac{p (\lambda_1-K_3-\frac{2p-1}{2}K_4^2)}{(1+2\lambda_1)^p}>0$ with $\tau\le \frac{\lambda_1-K_3-\frac{2p-1}{2}K_4^2}{2K_3^2}$.
Iterating \eqref{mall-contract-p} from $j=m+1$ to $j=k-1$, and using \eqref{DM-p}, \eqref{G-lip}, and \eqref{xnm-1+} with $\gamma=0$, gives
\begin{align*}
\|\DD_sX_k^N\|_{L_\omega^{2p}\LL_2^0}^{2p}
&\le C(1-c\tau)^{k-m-1} \|S_{N,\tau}\PP_NG(X_m^N)\|_{L_\omega^{2p}\LL_2^0}^{2p} \\
&\le Ce^{-C(t_k-[s]_\tau)}(1+\|X_0\|_1^{2p}),
\end{align*}
which proves \eqref{DM-pointwise-p}.

It remains to estimate the interpolation. For $t \in [t_k,t_{k+1})$, differentiating \eqref{xhat} gives
\begin{align}\label{DM-hat-column-p}
\DD_r^l\hat X^N(t)
= \begin{cases}
& S_{N,\tau}\PP_NG(X_k^N)\widetilde g_l, \quad r \in [t_k,t],\\
&\DD_r^lX_k^N +(t-t_k)A_NS_{N,\tau}\DD_r^lX_k^N +(t-t_k)S_{N,\tau}\PP_NDF_\tau(X_k^N)\DD_r^lX_k^N\\
&\quad +S_{N,\tau}\PP_NDG(X_k^N)(\DD_r^lX_k^N)(W(t)-W(t_k)), \quad 0\le r<t_k.
\end{cases}
\end{align}
If $r \in [t_k,t]$, then \eqref{G1} and \eqref{xnm-1+} with $\gamma=0$ imply
\begin{align*}
\|\DD_r\hat X^N(t)\|_{L_\omega^{2p}\LL_2^0}
\le C(1+\|X_0\|_1)
\le Ce^{-C(t-[r]_\tau)}(1+\|X_0\|_1).
\end{align*}
If $0\le r<t_k$, then \eqref{DM-hat-column-p}, \eqref{ftau'}, \eqref{DG}, and the BDG inequality yield
\begin{align*}
\|\DD_r\hat X^N(t)\|_{L_\omega^{2p}\LL_2^0}
\le C\|\DD_rX_k^N\|_{L_\omega^{2p}\LL_2^0}.
\end{align*}
Combining this with \eqref{DM-pointwise-p} gives \eqref{mall-est-local-p}.
\end{proof}

\begin{cor}\label{cor-mall-theta}
Let $p \in \nn_+$, $X_0 \in \dot H^{1+\gamma}$ with $\gamma\in(0,1)$ satisfying \eqref{gamma}, $\theta \in (0,\gamma)$ , and Assumptions \ref{ap-f}-\ref{ap-G} hold with $\lambda_1>K_3+\frac12\max\{(2p(q+1)-1)K_4^2,\,(2p(q+1)^2-1)K_6\}$.
There exist constants $C,c>0$ and $\tau_{\max} \in (0,1)$ such that for any $k \in \nn_+$, $0\le s\le t_k$, $t \in [t_k,t_{k+1})$, $\tau \in (0,\tau_{\max}]$, and $0\le r\le t$,
\begin{align}
\|\DD_sX_k^N\|_{L_\omega^{2p}\LL_2^\theta}
& \le Ce^{-c(t_k-[s]_\tau)} (1+\|X_0\|_{1+\gamma}^{q(q+1)+1}), \label{DM-pointwise-theta} \\ 
\|\DD_r \hat X^N(t)\|_{L_\omega^{2p}\LL_2^\theta}
& \le Ce^{-c(t-[r]_\tau)} (1+\|X_0\|_{1+\gamma}^{q(q+1)+1}). \label{mall-est-local-theta}
\end{align}
\end{cor}

\begin{proof}
Iterating \eqref{DM-rec-p} gives
\begin{align*}
Y_k
&=S_{N,\tau}^{k-m}\PP_NG(X_m^N)
+\tau\sum_{j=m+1}^{k-1}S_{N,\tau}^{k-j}\PP_NDF_\tau(X_j^N)Y_j \notag\\
&\quad+\sum_{j=m+1}^{k-1}S_{N,\tau}^{k-j}\PP_NDG(X_j^N)(Y_j)\delta_jW
=:Y_k^{(0)}+Y_k^{(F)}+Y_k^{(G)}.
\end{align*}
By \eqref{G1}, \eqref{xnm-1+}, and \eqref{smo-pn},
\begin{align*}
\|Y_k^{(0)}\|_{L_\omega^{2p}\LL_2^\theta}
\le Ce^{-c(t_k-[s]_\tau)}(1+\|X_0\|_{1+\gamma}).
\end{align*}
Assumption \ref{ap-f}, H\"older inequality, \eqref{gamma}, \eqref{xnm-1+}, and \eqref{DM-pointwise-p} yield
\begin{align}\label{DF}
\|DF_\tau(X_j^N)Y_j\|_{L_\omega^{2p}\LL_2^0}
\le Ce^{-c(t_j-[s]_\tau)}(1+\|X_0\|_{1+\gamma}^{q(q+1)+1}).
\end{align}
Hence, by \eqref{smo-pn} and \eqref{disc-sum} with $\varepsilon=1-\theta/2$,
\begin{align*}
\|Y_k^{(F)}\|_{L_\omega^{2p}\LL_2^\theta}
&\le C\tau\sum_{j=m+1}^{k-1}
(t_k-t_j)^{-\theta/2}e^{-c(t_k-t_j)}
e^{-c(t_j-[s]_\tau)}
(1+\|X_0\|_{1+\gamma}^{q(q+1)+1})\\
&\le Ce^{-c(t_k-[s]_\tau)}
(1+\|X_0\|_{1+\gamma}^{q(q+1)+1}).
\end{align*}
By BDG inequality in $\LL_2^\theta$,
\eqref{smo-pn}, \eqref{DG}, \eqref{DM-pointwise-p}, and \eqref{disc-sum}
with $\varepsilon=1-\theta$,
\begin{align*}
\|Y_k^{(G)}\|_{L_\omega^{2p}\LL_2^\theta}^2
&\le C\tau\sum_{j=m+1}^{k-1}
(t_k-t_j)^{-\theta}e^{-c(t_k-t_j)}
\|Y_j\|_{L_\omega^{2p}\LL_2^0}^2\\
&\le Ce^{-c(t_k-[s]_\tau)}
(1+\|X_0\|_1)^2.
\end{align*}
Combining the last three estimates proves \eqref{DM-pointwise-theta}.

For the interpolation, if $r \in [t_k,t]$, then \eqref{DM-hat-column-p}, \eqref{G1}, and \eqref{xnm-1+} directly give \eqref{mall-est-local-theta}. 
If $0\le r<t_k$, then \eqref{DM-hat-column-p}, \eqref{smo-pn}, \eqref{DG}, and the BDG inequality yield
\begin{align*}
\|\DD_r\hat X^N(t)\|_{L_\omega^{2p}\LL_2^\theta}
&\le C\|\DD_rX_k^N\|_{L_\omega^{2p}\LL_2^\theta}
+C\tau^{1-\theta/2}\|DF_\tau(X_k^N)\DD_rX_k^N\|_{L_\omega^{2p}\LL_2^0} \\
&\quad+C\tau^{(1-\theta)/2}\|\DD_rX_k^N\|_{L_\omega^{2p}\LL_2^0}.
\end{align*}
Using \eqref{DM-pointwise-theta}, \eqref{DM-pointwise-p}, and \eqref{DF} concludes \eqref{mall-est-local-theta}.
\end{proof}

\section{Proof of UIT Weak Error Estimate}
\label{sec4}
In this section, we prove the UIT weak error estimate in Theorem \ref{tm-weak}, which in turn yields the ergodic error estimate in Corollary \ref{cor-erg}. 
Let $\rho\in(0,1)$ be arbitrary. Throughout this section, choose $\varrho\in(1/4,1/2)$ and set $\chi=0$ when $\gamma=0$, while choosing $\chi\in(0,\gamma)$ when $\gamma>0$, such that $\rho<1/2+\varrho$ and $\chi+2\varrho>d/2$.
We also fix $0<\varepsilon<\min\{1/2, 1/2+\varrho-\rho\}$ and $0<\delta<\min\{1/2-\kappa_{d,\varrho,\chi}, \varrho-\kappa_{d,\chi}\}$.

Fix $m\ge2$ and $\varphi\in\CC_b^3(H)$.
 By the mean-value theorem and \cite[Theorem~4.1]{LS25},
\begin{align*}
    |\ee[\varphi(X(t_m))-\varphi(X^M(t_m))]|
    \le C(t_m,X_0)\lambda_M^{-\frac{1+\gamma}2}\to 0,
    \quad \text{as} \quad M\to \infty.
\end{align*}
It remains to estimate $|\ee[\varphi(X^M(t_m))-\varphi(X_m^N)]|$.
We split this term as
\begin{align*}
	&|\ee [\varphi(X^M(t_m))-\varphi(X^N_m)]| 
	=|\ee U_{t_m}^M(X^M_0)-\ee U_{0}^M(X^N_m)|\\
	& \le  |\ee [U_{t_m}^M(X^M_0) - U_{t_m}^M(X^N_0)]|
	+|\ee [U_{t_m}^M(X^N_0) - U_{0}^M(X^N_m)]|.
\end{align*}
Before estimating the two terms on the right-hand side, we record a direct consequence of Lemma~\ref{lm-xn} and the embedding \eqref{gamma}: for $j=1,2,3$,
\begin{align}
\Psi_t^{(j)}(x) \le C\bigl(1+\|x\|_{1+\gamma}^{r_j(q+1)}\bigr),
\qquad x \in \dot H^{1+\gamma},\ t\ge0,
\label{phi}
\end{align}
where $(r_1,r_2,r_3)=(q,4q-1,8q-2)$.
Hence, by \eqref{Du}, H\"older inequality, \eqref{phi}, and \eqref{pn}, for any $\alpha \in [0,1/2)$,
\begin{align*}
&\quad |\ee [U_{t_m}^M(X^M_0) - U_{t_m}^M(X^N_0)]| \\ 
& \le C(1+t_m^{-\alpha})e^{- ct_m} \int_0^1 \ee [\Psi_{t_m}^{(1)}(\theta X^M_0+(1-\theta)X^N_0)
\|({\rm Id} - \PP_N)X^M_0\|_{-2 \alpha}] \dd \theta \\
& \le C(1+t_m^{-\alpha}) (1+\|X^M_0\|_{L_\omega^{2q(q+1)}\dot H^{1+\gamma}}^{q(q+1)}+\|X^N_0\|_{L_\omega^{2q(q+1)}\dot H^{1+\gamma}}^{q(q+1)}) 
\|({\rm Id} - \PP_N)X^M_0\|_{L_\omega^2 \dot H^{-2 \alpha}} \\ 
& \le  C (1+t_m^{-\alpha}) (1+\|X_0\|_{1+\gamma}^{q(q+1)+1}) \lambda_N^{-\frac{1+\gamma}{2}-\alpha}.
\end{align*}
For the second term, using the continuous interpolation \eqref{xhat}, write
\begin{align*}
|\ee[U_{t_m}^M(X_0^N)-U_0^M(X_m^N)]|
&\le |\ee[U_{t_m}^M(X_0^N)-U_{t_m-t_1}^M(X_1^N)]|\\
&\quad+\sum_{i=1}^{m-1} |\ee[U_{t_m-t_{i+1}}^M(\hat X^N(t_{i+1})) -U_{t_m-t_i}^M(\hat X^N(t_i))]|.
\end{align*}
For $1\le i\le m-1$, It\^o formula, \eqref{kol}, and \eqref{xhat} give
\begin{align*}
& \ee [U_{t_m-t_{i+1}}^M(\hat X^N(t_{i+1}))-U_{t_m-t_{i}}^M(\hat X^N(t_{i}))] \\
&= \ee \int_{t_i}^{t_{i+1}} \Big[-\partial_t U_{t_m-t}^M(\hat X^N)
+ DU_{t_m-t}^M(\hat X^N)(A_N   S_{N,\tau} X_i^N+S_{N,\tau}\PP_NF_\tau(X_i^N)) \\
& \quad + \frac12 \sum_{k \in \nn_+} D^2U_{t_m-t}^M(\hat X^N)(S_{N,\tau}\PP_NG(X_i^N)Q^{\frac12} g_k, S_{N,\tau}\PP_NG(X_i^N)Q^{\frac12} g_k)\Big] \dd t
=: \sum_{j=1}^3 I_j^i, 
\end{align*}
where 
\begin{align*}
I_1^i & := \ee \int_{t_i}^{t_{i+1}}DU_{t_m-t}^M(\hat X^N) (A_N  S_{N,\tau} X_i^N-A  \hat X^N)\dd t, \\
I_2^i & := \ee \int_{t_i}^{t_{i+1}}DU_{t_m-t}^M(\hat X^N)( S_{N,\tau}\PP_NF_\tau(X_i^N)-\PP_MF(\hat X^N))\dd t \\
I_3^i & := \frac12 \sum_{k \in \nn_+} \ee \int_{t_i}^{t_{i+1}} [D^2U_{t_m-t}^M(\hat X^N) (S_{N,\tau}\PP_NG(X_i^N)Q^{\frac12} g_k, S_{N,\tau}\PP_NG(X_i^N)Q^{\frac12} g_k)\\
& \qquad  \qquad - D^2U_{t_m-t}^M(\hat X^N) (\PP_MG(\hat X^N)Q^{\frac12} g_k, \PP_MG(\hat X^N)Q^{\frac12} g_k)] \dd t.
\end{align*}
We estimate $I_1^i$, $I_2^i$, and $I_3^i$ successively.
To this end, let us denote 
\begin{align} \label{phi-a}
\phi_\alpha(t):=(1+t^{-\alpha})e^{-ct}, \quad \alpha\in[0,1),~ t>0.
\end{align}

\begin{lm} \label{lm-u}
Let Assumptions~\ref{ap-f}--\ref{ap-G} hold with $\lambda_1>K_3+\Lambda_{\rm w}$. For any $\alpha\in[0,1/2)$, there exist constants $C>0$ and $\tau_{\max}\in(0,1)$ such that, for any $\tau\in(0,\tau_{\max}]$ and $m\ge2$,
\begin{align}
|\ee[U_{t_m}^M(X_0^N)-U_{t_m-t_1}^M(X_1^N)]|
&\le C \phi_\alpha(t_m-t_1) (1+\|X_0\|_{1+\gamma}^{(q+1)^3})
(\lambda_N^{-\frac{1+\gamma}{2}-\alpha}
+\tau^{\frac12+\alpha}).
\label{U1}
\end{align}
\end{lm}

\begin{proof}
For convenience, we denote $Z_\theta:=\theta X^M(t_1,X_0^N)+(1-\theta)X_1^N$, $\theta\in[0,1]$.
The Markov property, mean value theorem, \eqref{Du}, and \eqref{phi} yield
\begin{align*}
&|\ee[U_{t_m}^M(X_0^N)-U_{t_m-t_1}^M(X_1^N)]|\\
&=|\ee U_{t_m-t_1}^M(X^M(t_1,X_0^N))-\ee U_{t_m-t_1}^M(X_1^N)|\\
&\le\int_0^1\ee|DU_{t_m-t_1}^M(Z_\theta)(X^M(t_1,X_0^N)-X_1^N)|\dd\theta\\
&\le C\phi_\alpha(t_m-t_1)
\left(1+\sup_{0\le\theta\le1}\|Z_\theta\|_{L_\omega^{2q(q+1)}\dot H^{1+\gamma}}^{q(q+1)}\right)
\|X^M(t_1,X_0^N)-X_1^N\|_{L_\omega^2\dot H^{-2\alpha}}\\
&\le C\phi_\alpha(t_m-t_1)
\left(1+\|X^M(t_1,X_0^N)\|_{L_\omega^{2q(q+1)}\dot H^{1+\gamma}}^{q(q+1)}
+\|X_1^N\|_{L_\omega^{2q(q+1)}\dot H^{1+\gamma}}^{q(q+1)}\right)\\
&\quad\times\left(
\|({\rm Id}-\PP_N)X^M(t_1,X_0^N)\|_{L_\omega^2\dot H^{-2\alpha}}
+\|\PP_NX^M(t_1,X_0^N)-X_1^N\|_{L_\omega^2\dot H^{-2\alpha}}
\right).
\end{align*}
Then we conclude \eqref{U1} by \eqref{xn}, \eqref{xnm-1+}, \eqref{one-err}, and \eqref{pn}.
\end{proof}

\begin{lm} \label{lm-I1}
Let Assumptions \ref{ap-f}-\ref{ap-G} hold with $\lambda_1 > K_3 +\Lambda_{\rm w}$. 
There exist constants $C>0$ and $\tau_{\max}\in(0,1)$ such that for any $\tau\in(0,\tau_{\max}]$ and $m \ge 2$,
\begin{align} \label{I1-est}
\sum_{i=1}^{m-1}|I_1^i|
\le C(1+t_m^{-\varrho})
(1+\|X_0\|_{1+\gamma}^{4q(q+1)^2+1})
\tau^{\frac12+\varrho-\varepsilon}.
\end{align}
\end{lm}

\begin{proof}
Let $1 \le i \le m-1$.
The definition \eqref{xhat} of $\hat X$ leads to 
\begin{align*}
	I_1^i & = \int_{t_i}^{t_{i+1}}\ee [DU_{t_m-t}^M(\hat X^N)(A_N [S_{N,\tau} - {\rm Id} ] X_i^N)] \dd t \\
	& \quad - \int_{t_i}^{t_{i+1}}\ee [DU_{t_m-t}^M(\hat X^N) ([t-t_i] A_N ^2S_{N,\tau} X_i^N)] \dd t \\
	&\quad - \int_{t_i}^{t_{i+1}}\ee [DU_{t_m-t}^M(\hat X^N) ([t-t_i] A_N  S_{N,\tau} \PP_N F_\tau(X_i^N))] \dd t \\
	&\quad - \int_{t_i}^{t_{i+1}}\ee [DU_{t_m-t}^M(\hat X^N) (A_N S_{N,\tau} \PP_N G(X_i^N) [W(t)-W(t_i)])] \dd t 
	=: \sum_{j=1}^4 I_{1j}^i.
\end{align*}

\emph{Estimate of $I_{11}^i$.}
Using the identity $S_{N,\tau} - {\rm Id}  = \tau A_N S_{N,\tau}$ and \eqref{full-sum}, we can decompose the first term $I_{11}^i$ into three terms:
\begin{align*}
    I_{11}^i
    &= \tau \int_{t_i}^{t_{i+1}} \ee [DU_{t_m-t}^M(\hat X^N) ((-A_N)^2 S_{N,\tau}^{i+1} X^N_0)] \dd t  \\
    &\quad + \tau^2 \int_{t_i}^{t_{i+1}} \ee [DU_{t_m-t}^M(\hat X^N) ((-A_N)^2 \sum_{j=0}^{i-1} S_{N,\tau}^{i+1-j} \PP_NF_\tau(X^N_j) )] \dd t  \\
    &\quad + \tau \int_{t_i}^{t_{i+1}} \ee [DU_{t_m-t}^M(\hat X^N)  ((-A_N)^2 \sum_{j=0}^{i-1} S_{N,\tau}^{i+1-j} \PP_NG(X^N_j)\delta_j W)] \dd t 
    =: \sum_{j=1}^3 I_{11j}^i.
\end{align*} 
By Corollary~\ref{cor-DU} and a distribution of fractional powers of $-A$, we obtain
\begin{align*}
|I_{111}^i|+|I_{112}^i|
& \le C\tau \int_{t_i}^{t_{i+1}} \phi_\varrho(t_m-t) \|(-A)^{\frac12-\varrho+\varepsilon}S_{N,\tau}\|_{\LL}  \\
&\quad \times \Big( \|(-A)^{1-\varepsilon}S_{N,\tau}^{i}\|_{\LL}
\ee [\Psi_{t_m-t}^{(1)}(\hat X^N)\|X^N_0\|_1]  \\
&\qquad\quad
+\tau\sum_{j=0}^{i-1}
\|(-A)^{1-\varepsilon}S_{N,\tau}^{i-j}\|_{\LL}
\ee[\Psi_{t_m-t}^{(1)}(\hat X^N)\|F_\tau(X_j^N)\|_1] \Big) \dd t .
\end{align*}
Then, using \eqref{smo-pn}, \eqref{phi}, \eqref{F-bounds-0}, H\"older inequality, \eqref{xhat-est}, \eqref{xnm-1+}, and \eqref{disc-sum}, we obtain
\begin{align*}
|I_{111}^i|+|I_{112}^i|
& \le C\tau^{1/2+\varrho-\varepsilon}
(1+\|X_0\|_{1+\gamma}^{(q+1)^3})  
 \int_{t_i}^{t_{i+1}}
\phi_\varrho(t_m-t)
\big(1+t_i^{-1+\varepsilon}e^{-ct_i}\big)\dd t .
\end{align*}
Next, we estimate $I_{113}^i$ by a Malliavin argument.
Let $a_i:=(t_i-1)_+$.
Since
$$\sum_{j=0}^{i-1}S_{N,\tau}^{i+1-j}\PP_NG(X^N_j)\delta W_j
= \int_0^{t_i}S_{N,\tau}^{i+1-\ell(s)}\PP_NG(X_{\ell(s)}^N) \dd W(s),$$ 
we split $I_{113}^i$ as 
\begin{align*}
	|I_{113}^i|
	& \le C \tau \Big| \int_{t_i}^{t_{i+1}}\ee \Big[DU_{t_m-t}^M(\hat X^N) \Big( (-A_N)^2 \int_{0}^{a_i}\PP_N S_{N,\tau}^{i+1-\ell(s)}G(X_{\ell(s)}^N) \dd W(s) \Big) \Big] \dd t\Big|\\
	&\quad+C\tau\Big| \int_{t_i}^{t_{i+1}}\ee \Big[DU_{t_m-t}^M(\hat X^N) \Big((-A_N)^2 \int_{a_i}^{t_i}\PP_N S_{N,\tau}^{i+1-\ell(s)}G(X_{\ell(s)}^N) \dd W(s)\Big) \Big] \dd t\Big|\\
	&=:I_{1131}^i+I_{1132}^i.
\end{align*}
By Cauchy--Schwarz inequality, It\^o isometry, \eqref{Du} with
$\alpha=0$, \eqref{phi}, \eqref{G1}, and \eqref{noise-smo-est}, we get
\begin{align*}
I_{1131}^i
& \le  C\tau \int_{t_i}^{t_{i+1}}\Big( \int_{0}^{a_i}\ee\|(-A_N)^2 \PP_N S_{N,\tau}^{i+1-\ell(s)}G(X_{\ell(s)}^N)\|_{\LL_2^0}^2  \dd s\Big)^{\frac12}\\
&\quad \times e^{- c (t_m-t)}\|\Psi_{t_m-t}^{(1)}(\hat X^N)\|_{L_\omega^2}\dd t\\
& \le  C \tau (1+\|X_0\|_{1+\gamma}^{q(q+1)^2}) \int_{t_i}^{t_{i+1}}e^{- c (t_m-t)}\\
&\quad \times\Big( \int_{0}^{a_i} \|(-A_N)^{\frac32} \PP_N S_{N,\tau}^{i+1-\ell(s)}\|_\LL^2 
	\|G(X_{\ell(s)}^N)\|_{L_\omega^2 \LL_2^1}^2 \dd s\Big)^{\frac12}  \dd t\\ 
& \le C\tau(1+\|X_0\|_{1+\gamma}^{(q+1)^3}) \int_{t_i}^{t_{i+1}}e^{-c(t_m-t)} \dd t .
\end{align*}
Applying the Malliavin IBP formula  \eqref{ibp} and \eqref{D2u-mix} with $\vartheta=\varrho$, we obtain
\begin{align*}
&\quad I_{1132}^i\\
&\le C\tau \int_{t_i}^{t_{i+1}} \int_{a_i}^{t_i} \sum_{l \in \nn_+}\ee|D^2U_{t_m-t}^M(\hat X^N) (\DD_s^l\hat X^N,(-A_N)^2S_{N,\tau}^{i+1-\ell(s)} \PP_NG(X_{\ell(s)}^N)\widetilde g_l)|\dd s\dd t \\
&\le C\tau \int_{t_i}^{t_{i+1}}\phi_\varrho(t_m-t) \int_{a_i}^{t_i}\ee[\Psi_{t_m-t}^{(2)}(\hat X^N) \|\DD_s\hat X^N\|_{\LL_2^\chi} \|(-A)^{1-\varepsilon}S_{N,\tau}^{i-\ell(s)}\|_{\LL} \\
&\qquad\times \|(-A)^{1/2-\varrho+\varepsilon}\PP_NS_{N,\tau}\|_{\LL} \|G(X_{\ell(s)}^N)\|_{\LL_2^1}]\dd s\dd t .
\end{align*}
Using \eqref{smo-pn}, H\"older  inequality, \eqref{G1}, and \eqref{disc}, together with \eqref{mall-est-local-p} with $p=2$ for $\gamma=\chi=0$ and \eqref{mall-est-local-theta} with $p=2$ and $\theta=\chi$ for $\gamma>0$, we conclude that
\begin{align*}
I_{1132}^i
&\le C\tau^{\frac12+\varrho-\varepsilon} \int_{t_i}^{t_{i+1}}\phi_\varrho(t_m-t) \int_{a_i}^{t_i}(t_i-[s]_\tau)^{-1+\varepsilon} e^{-c(t_i-[s]_\tau)} \|\Psi_{t_m-t}^{(2)}(\hat X^N)\|_{L_\omega^2} \\ 
&\qquad\times \|\DD_s\hat X^N\|_{L_\omega^4\LL_2^\chi} \|G(X_{\ell(s)}^N)\|_{L_\omega^4\LL_2^1}\dd s\dd t \\
&\le C\tau^{\frac12+\varrho-\varepsilon} (1+\|X_0\|_{1+\gamma}^{4q(q+1)^2}) \int_{t_i}^{t_{i+1}}\phi_\varrho(t_m-t)\dd t .
\end{align*}
The above analysis, together with \eqref{ele-int}, yields 
\begin{align} \label{i11}
	\sum_{i=1}^{m-1}|I_{11}^i| \le C(1+t_m^{-\varrho}) \tau^{\frac12+\varrho-\varepsilon} (1+\|X_0\|_{1+\gamma}^{4q(q+1)^2}). 
\end{align}  

\emph{Estimate of $I_{12}^i$, $I_{13}^i$, and $I_{14}^i$.}
The term $I_{12}^i$ has the same structure as $I_{11}^i$, up to the additional factor $t-t_i\le\tau$. Thus, repeating the preceding argument gives
\begin{align}\label{i12}
\sum_{i=1}^{m-1}|I_{12}^i|
\le C(1+t_m^{-\varrho})\tau^{\frac12+\varrho-\varepsilon} (1+\|X_0\|_{1+\gamma}^{4q(q+1)^2}).
\end{align} 
For $I_{13}^i$, by \eqref{Du}, a distribution of fractional powers of $-A$ and \eqref{smo-pn}, we have
\begin{align*}
|I_{13}^i|
& \le C\tau \int_{t_i}^{t_{i+1}} \phi_\varrho(t_m-t) \|(-A)^{\frac12-\varrho}S_{N,\tau}\|_{\LL}
\ee [\Psi_{t_m-t}^{(1)}(\hat X^N)\|F_\tau(X_i^N)\|_1] \dd t  \\
& \le C\tau^{\frac12+\varrho}
 \int_{t_i}^{t_{i+1}}
\phi_\varrho(t_m-t)
\|\Psi_{t_m-t}^{(1)}(\hat X^N)\|_{L_\omega^2}
\|F_\tau(X_i^N)\|_{L_\omega^2\dot H^1} \dd t .
\end{align*}
Then H\"older inequality, \eqref{phi}, \eqref{xhat-est}, \eqref{F-bounds-0}, and \eqref{xnm-1+} yield
\begin{align}
\sum_{i=1}^{m-1}|I_{13}^i|
\le C\tau^{\frac12+\varrho}
 (1+\|X_0\|_{1+\gamma}^{(q+1)^3} ).
\label{i13}
\end{align}

Before estimating the last term $I_{14}^i$, we introduce $Z_\theta^i := \theta \hat X^N(t) + (1-\theta)X_i^N$ for $\theta \in [0, 1]$.
Using the martingale property and the mean value theorem, we can rewrite $I_{14}^i$ as  
\begin{align*}
	I_{14}^i 
&=\ee \int_{t_i}^{t_{i+1}} \int_0^1 D^2U_{t_m-t}^M(Z_\theta^i) 
(A_N S_{N,\tau} \PP_N G(X_i^N) [W(t)-W(t_i)], \hat X^N-  X_i^N)\dd \theta \dd t.
\end{align*} 
By \eqref{D2u-mix} and H\"older inequality, we have
\begin{align*}
|I_{14}^i|
& \le C \int_{t_i}^{t_{i+1}} \int_0^1\phi_\varrho(t_m-t)
\ee[\Psi_{t_m-t}^{(2)}(Z_\theta^i)\|\hat X^N- X_i^N\|_\chi \\
&\qquad \times \|(-A)^{1-\varrho}S_{N,\tau}\PP_NG(X_i^N)[W(t)-W(t_i)]\|]\dd\theta\dd t\\
& \le C \int_{t_i}^{t_{i+1}}\phi_\varrho(t_m-t)
 \int_0^1 \|\Psi_{t_m-t}^{(2)}(Z_\theta^i)\|_{L_\omega^2}
\|\hat X^N-X_i^N\|_{L_\omega^4\dot H^\chi}\\
&\qquad \times
\|(-A)^{1-\varrho}S_{N,\tau}\PP_NG(X_i^N)[W(t)-W(t_i)]\|_{L_\omega^4L_\xi^2}
\dd\theta\dd t .
\end{align*}
The BDG inequality, \eqref{smo-pn}, and \eqref{G1} yield
\begin{align*}
& \|(-A)^{1-\varrho}S_{N,\tau} \PP_N G(X_i^N) [W(t)-W(t_i)]\|_{L_\omega^4 L_\xi^2} 
\le C\tau^\varrho(1+\|X_i^N\|_{L_\omega^4 \dot H^1}).
\end{align*}
Hence, by the UIT estimates \eqref{xhat-est} and \eqref{xnm-1+}, we obtain  
\begin{align} \label{i14}
    \sum_{i=1}^{m-1}|I_{14}^i| \le C \tau^{\frac12+\varrho}(1+\|X_0\|_{1+\gamma}^{4q(q+1)^2+1}).
\end{align}
Then we conclude \eqref{I1-est} by the estimates \eqref{i11}, \eqref{i12}, \eqref{i13}, and \eqref{i14}.
\end{proof}

\begin{lm} \label{lm-I2}
Let Assumptions \ref{ap-f}-\ref{ap-G} hold with $\lambda_1 > K_3 +\Lambda_{\rm w}$. 
There exist constants $C>0$ and $\tau_{\max}\in(0,1)$ such that for any $\tau\in(0,\tau_{\max}]$,
\begin{align} \label{I2-est}
\sup_{m \ge2} \sum_{i=1}^{m-1} |I_2^i| 
\le C (1+\|X_0\|_{1+\gamma}^{4q(q+1)^2+1})
(\lambda_N^{-\frac{1+\gamma}{2}-\varrho}+\tau^{\frac12+\varrho}).
\end{align}
\end{lm}

\begin{proof}
For $I_2^i$, we split it into four terms:
\begin{align*}
    I_2^i
    &= \ee \int_{t_i}^{t_{i+1}} DU_{t_m-t}^M(\hat X^N) ( [S_{N,\tau} - {\rm Id} ] \PP_N F_\tau (X_i^N) ) \dd t  \\
    &\quad + \ee \int_{t_i}^{t_{i+1}} DU_{t_m-t}^M(\hat X^N) ( \PP_N [F_\tau (X_i^N) - F (X_i^N)] ) \dd t \\
    &\quad + \ee \int_{t_i}^{t_{i+1}} DU_{t_m-t}^M(\hat X^N) ( \PP_N [F (X_i^N) - F(\hat X^N)] ) \dd t \\
    &\quad + \ee \int_{t_i}^{t_{i+1}} DU_{t_m-t}^M(\hat X^N) ( \PP_N F(\hat X^N) - \PP_M F(\hat X^N) ) \dd t 
        =: \sum_{j=1}^4 I_{2j}^i.
\end{align*}

\emph{Estimate of $I_{21}^i$, $I_{22}^i$, and $I_{24}^i$.}
For the term $I_{21}^i$, by \eqref{Du}, H\"older inequality, the estimates \eqref{phi}, \eqref{F-bounds-0}, \eqref{xnm-1+}, \eqref{sn}, and \eqref{ele-int}, we have
\begin{align}\label{i21}
    \sum_{i=1}^{m-1}|I_{21}^i|
    & \le C \tau^{\frac12+\varrho} \sum_{i=1}^{m-1} \int_{t_i}^{t_{i+1}} \phi_\varrho(t_m-t)
        \ee [ \Psi_{t_m-t}^{(1)}(\hat X^N)  \|F_\tau (X_i^N) \|_1 ] \dd t \notag\\
    & \le C \tau^{\frac12+\varrho} (1+\|X_0\|_{1+\gamma}^{(q+1)^3}).
\end{align}
For $I_{22}^i$, by \eqref{Du}, H\"older inequality, \eqref{phi}, \eqref{f-ftau}, the Sobolev embedding $\dot H^{1+\gamma}\hookrightarrow L_\xi^{6q+2}$, and \eqref{xnm-1+}, we derive
\begin{align}\label{i22}
    \sum_{i=1}^{m-1}|I_{22}^i|
    & \le C\tau \sum_{i=1}^{m-1} \int_{t_i}^{t_{i+1}} e^{- c (t_m-t)} \ee [ \Psi_{t_m-t}^{(1)}(\hat X^N) ( 1+\|X_i^N\|_{L_\xi^{6q+2}}^{3q+1} )] \dd t\notag\\
    & \le C\tau (1+\|X_0\|_{1+\gamma}^{(q+1)(q+2)^2}).
\end{align} 
For $I_{24}^i$, since $\dot H^{1+\gamma}$ is an algebra for
$\gamma$ satisfying \eqref{gamma}, Assumption~\ref{ap-f} gives
$\|F(v)\|_{1+\gamma} 
\le C(1+\|v\|_{1+\gamma}^{q+1})$,
$v\in\dot H^{1+\gamma}.$
Thus, by \eqref{Du}, \eqref{pn}, H\"older's inequality,
\eqref{phi}, and \eqref{xhat-est},
\begin{align}\label{i24}
\sum_{i=1}^{m-1}|I_{24}^i|
&\le C\lambda_N^{-\frac{1+\gamma}{2}-\varrho}
\sum_{i=1}^{m-1}\int_{t_i}^{t_{i+1}}
\phi_\varrho(t_m-t)
\ee[\Psi_{t_m-t}^{(1)}(\hat X^N)
(1+\|\hat X^N\|_{1+\gamma}^{q+1})]\dd t \notag\\
&\le C(1+\|X_0\|_{1+\gamma}^{(q+1)^3})
\lambda_N^{-\frac{1+\gamma}{2}-\varrho}.
\end{align}
\emph{Estimate of $I_{23}^i$.}
For $I_{23}^i$, Taylor formula yields
\begin{align*}
    I_{23}^i &= -\ee \int_{t_i}^{t_{i+1}} (t-t_i) DU_{t_m-t}^M(\hat X^N) ( \PP_N DF(X_i^N) A_N S_{N,\tau} X_i^N ) \dd t \\
    &\quad - \ee \int_{t_i}^{t_{i+1}} (t-t_i) DU_{t_m-t}^M(\hat X^N) ( \PP_N DF(X_i^N) S_{N,\tau} \PP_N F_\tau(X_i^N) ) \dd t \\
    &\quad - \ee \int_{t_i}^{t_{i+1}} DU_{t_m-t}^M(\hat X^N) \Big( \PP_N DF(X_i^N) \int_{t_i}^{t} S_{N,\tau}\PP_N G(X_i^N) \dd W(s) \Big) \dd t \\
    &\quad - \ee \int_{t_i}^{t_{i+1}} DU_{t_m-t}^M(\hat X^N) \Big( \PP_N \int_{0}^1 (1-\theta) D^2F(Z_\theta^i) (\hat X^N-X_i^N, \hat X^N-X_i^N) \dd \theta \Big) \dd t \\
    &=: \sum_{j=1}^4I_{23j}^i.
\end{align*}
Using \eqref{full-sum}, we further split $I_{231}^i = T_1^i + T_2^i + T_3^i$ as follows: 
\begin{align*}
    I_{231}^i &= -\ee \int_{t_i}^{t_{i+1}} (t-t_i) DU_{t_m-t}^M(\hat X^N) ( \PP_N DF(X_i^N) A_N S_{N,\tau}^{i+1} X^N_0 ) \dd t \\
    &\quad - \tau \ee \int_{t_i}^{t_{i+1}} (t-t_i) DU_{t_m-t}^M(\hat X^N) \Big( \PP_N DF(X_i^N) \sum_{j=0}^{i-1} A_N S_{N,\tau}^{i+1-j} \PP_N F_\tau(X^N_j) \Big) \dd t \\
   &\quad - \ee \int_{t_i}^{t_{i+1}} (t-t_i) DU_{t_m-t}^M(\hat X^N) \Big( \PP_N DF(X_i^N) \sum_{j=0}^{i-1} A_N S_{N,\tau}^{i+1-j} \PP_N G(X^N_j) \delta W_j \Big) \dd t.
\end{align*}
Utilizing \eqref{Du}, H\"older inequality, \eqref{F-bounds-1}, the embedding \eqref{gamma}, \eqref{xhat-est}, \eqref{xnm-1+}, the property \eqref{smo-pn}, and the estimate \eqref{ele-int}, $|T_1^i|$ can be bounded as
\begin{align*}
	|T_1^i|
	& \le  C\tau \int_{t_i}^{t_{i+1}} e^{- c (t_m-t)} \|(-A)^{1/2}S_{N,\tau}^{i+1}\|_\LL
	\ee[\Psi_{t_m-t}^{(1)}(\hat X^N) (1+\|X_i^N\|_{1+\gamma}^q)  \|X^N_0\|_1]\dd t\\
	& \le  C\tau t_{i+1}^{-\frac12} e^{-ct_{i+1}} \int_{t_i}^{t_{i+1}} e^{- c (t_m-t)}
	\|\Psi_{t_m-t}^{(1)}(\hat X^N)\|_{L_\omega^2}
	(1+\|X_i^N\|_{L^{2q}_\omega\dot H^{1+\gamma}}^q) \|X_0\|_1 \dd t\\
	& \le  C\tau (1+\|X_0\|_{1+\gamma}^{(q+1)^3}) t_{i+1}^{-\frac12} e^{-ct_{i+1}} 
	 \int_{t_i}^{t_{i+1}} e^{-c(t_m-t)} \dd t.
\end{align*}
Analogously, by \eqref{disc-sum} we obtain
\begin{align*}
	|T_2^i|
	& \le  C\tau^2 \int_{t_i}^{t_{i+1}} e^{- c (t_m-t)}
	\sum_{j=0}^{i-1}\|(-A)^{1/2+\varrho}S_{N,\tau}^{i-j}\|_\LL \|(-A)^{1/2-\varrho} S_{N,\tau}\|_\LL\\
	&\quad \times \ee[\Psi_{t_m-t}^{(1)}(\hat X^N ) (1+\|X_i^N\|_{1+\gamma}^q) \|F_\tau(X^N_j)\|]\dd t\\ 
	& \le  C\tau^{\frac12+\varrho}(1+\|X_0\|_{1+\gamma}^{(q+1)^3}) \int_{t_i}^{t_{i+1}} e^{- c(t_m-t)}\dd t. 
\end{align*}
For $T_3^i$, we argue as in the estimate of  $I_{113}^i$. 
Set
\begin{align*}
\Gamma_i(s):=A_NS_{N,\tau}^{i+1-\ell(s)}\PP_N G(X_{\ell(s)}^N)\mathbf1_{[0,t_i]}(s), 
\quad \Gamma_i^l(s):=\Gamma_i(s)\widetilde g_l,
\quad s \ge 0.
\end{align*} 
Then
\begin{align*}
	\sum_{j=0}^{i-1}A_NS_{N,\tau}^{i+1-j}\PP_N G(X^N_j)\delta W_j
= \int_0^{t_i}\Gamma_i(s)\dd W(s).
\end{align*} 
Moreover, \eqref{G1} and \eqref{smo-pn} imply, for $p>1$ and $\nu \in [0,1]$,
\begin{align}\label{Gamma_L2+}
\|\Gamma_i(s)\|_{L^p(\Omega;\LL_2^{\nu})}
& \le
C\|(-A)^{\frac{1+\nu}2}S_{N,\tau}^{i+1-\ell(s)}\|_{\LL}
\|G(X_{\ell(s)}^N)\|_{L^p(\Omega;\LL_2^1)}\notag \\
& \le C (1+\|X_{\ell(s)}^N\|_1) (t_{i+1}-[s]_\tau)^{-\frac{1+\nu}{2}} e^{-c(t_{i+1}-[s]_\tau)} .
\end{align}
By the Malliavin IBP formula \eqref{ibp} and the chain rule \eqref{chain}, we have
\begin{align*}
T_3^i 
&= - \int_{t_i}^{t_{i+1}}(t-t_i) \int_0^{t_i}\sum_{l \in \nn_+}\ee[ D^2U_{t_m-t}^M(\hat X^N)(\DD_s^l \hat X^N,\PP_N DF(X_i^N)\Gamma_i^l(s))]\dd s\dd t\\
&\quad - \int_{t_i}^{t_{i+1}}(t-t_i) \int_0^{t_i}\sum_{l \in \nn_+}\ee[ DU_{t_m-t}^M(\hat X^N)(\PP_N D^2F(X_i^N)(\DD_s^lX_i^N,\Gamma_i^l(s)))]\dd s\dd t\\
&=: T_{31}^i+T_{32}^i.
\end{align*}
For $T_{31}^i$, by \eqref{D2u} with
$(\vartheta_1,\vartheta_2)=(0,0)$ and H\"older inequality, we have
\begin{align*}
|T_{31}^i|
&\le C\tau \int_{t_i}^{t_{i+1}} \phi_{\kappa_{d,\varrho}+\delta}(t_m-t) \int_0^{t_i}\|\Psi_{t_m-t}^{(2)}(\hat X^N)\|_{L_\omega^2} \|\DD_s\hat X^N\|_{L_\omega^4\LL_2^0}\\
&\qquad\times \|DF(X_i^N)\Gamma_i(s)\|_{L_\omega^4\LL_2^0}\,\dd s\dd t .
\end{align*}
Moreover, by \eqref{F-bounds-1}, \eqref{gamma}, H\"older inequality, \eqref{Gamma_L2+}, and \eqref{xnm-1+}, we have
\begin{align*}
\|DF(X_i^N)\Gamma_i(s)\|_{L_\omega^4\LL_2^0} 
&\le C (1+\|X_i^N\|_{L_\omega^{8q}\dot H^{1+\gamma}}^q ) (1+\|X_{\ell(s)}^N\|_{L_\omega^8\dot H^1} )\\
&\qquad\times (t_{i+1}-[s]_\tau)^{-1/2} e^{-c(t_{i+1}-[s]_\tau)}.
\end{align*}
Hence, by \eqref{phi}, \eqref{xhat-est}, \eqref{mall-est-local-p},
\eqref{xnm-1+}, and \eqref{disc},
\begin{align*}
|T_{31}^i|
&\le C\tau (1+\|X_0\|_{1+\gamma}^{4q(q+1)^2+1} ) \int_{t_i}^{t_{i+1}}\phi_{\kappa_{d,\varrho}+\delta}(t_m-t) \dd t . 
\end{align*}
For $T_{32}^i$, choose $\nu \in [0,\min\{1,d/2\})$ such that $\nu+2\varrho>d/2$. By \eqref{Du}, \eqref{D2F-d}, \eqref{gamma}, and H\"older inequality,
\begin{align*}
|T_{32}^i|
&\le C\tau \int_{t_i}^{t_{i+1}}\phi_\varrho(t_m-t)
 \int_0^{t_i}\|\Psi_{t_m-t}^{(1)}(\hat X^N)\|_{L_\omega^4}
\left(1+\|X_i^N\|_{L_\omega^{8(q-1)}
\dot H^{1+\gamma}}^{q-1}\right)\\
&\qquad\times
\|\DD_sX_i^N\|_{L_\omega^2\LL_2^0}
\|\Gamma_i(s)\|_{L_\omega^8\LL_2^\nu} \dd s\dd t .
\end{align*}
Using \eqref{phi}, \eqref{xhat-est}, \eqref{DM-pointwise-p}, \eqref{Gamma_L2+}, \eqref{xnm-1+}, and \eqref{disc} with exponent $(1-\nu)/2$, we obtain
\begin{align*}
|T_{32}^i|
&\le C\tau
\left(1+\|X_0\|_{1+\gamma}^{(q+1)^3}\right)
 \int_{t_i}^{t_{i+1}}\phi_\varrho(t_m-t)\,\dd t . 
\end{align*}
Combining these two estimates with the estimates of $T_1^i$ and $T_2^i$, and using \eqref{ele-int}, we get
\begin{align*}
\sum_{i=1}^{m-1}|I_{231}^i|
&\le C\tau^{1/2+\varrho}
(1+\|X_0\|_{1+\gamma}^{4q(q+1)^2+1}).
\end{align*}
For $I_{232}^i$, by \eqref{Du} with $\alpha=0$, H\"older inequality,
\eqref{F-bounds-1}, \eqref{gamma}, \eqref{f-grow}, \eqref{phi},
\eqref{xhat-est}, and \eqref{xnm-1+}, we obtain
\begin{align*}
\sum_{i=1}^{m-1}|I_{232}^i|
&\le C\tau\sum_{i=1}^{m-1} \int_{t_i}^{t_{i+1}}e^{-c(t_m-t)} \|\Psi_{t_m-t}^{(1)}(\hat X^N)\|_{L_\omega^4} (1+\|X_i^N\|_{L_\omega^{4q}\dot H^{1+\gamma}}^q) \notag\\
&\qquad\times (1+\|X_i^N\|_{L_\omega^{2q+2}}^{q+1})\dd t \notag\\
&\le C\tau(1+\|X_0\|_{1+\gamma}^{(q+1)^3}).
\end{align*} 
For $I_{233}^i$, the term with $DU_{t_m-t}^M(X_i^N)$ vanishes after conditioning on $\FFF_{t_i}$. Hence, the mean value formula, \eqref{D2u} with $(\vartheta_1,\vartheta_2)=(0,0)$, H\"older inequality, \eqref{F-bounds-1}, and \eqref{gamma} yield
\begin{align*}
|I_{233}^i| 
& \le C \int_{t_i}^{t_{i+1}} \int_0^1 \phi_{\kappa_{d,\varrho}+\delta}(t_m-t) \|\Psi_{t_m-t}^{(2)}(Z_\theta^i)\|_{L_\omega^2} \|\hat X^N- X_i^N\|_{L_\omega^4 L_\xi^2}  \\ 
	 &\quad (1+\|X_i^N\|_{L_\omega^{8q}\dot H^{1+\gamma}}^q) \| \int_{t_i}^{t} S_{N,\tau}\PP_NG(X_i^N) \dd W(s)\|_{L_\omega^8 L_\xi^2} \dd \theta\dd t .
\end{align*}
By the BDG inequality, \eqref{G-lip}, \eqref{phi}, \eqref{xhat-est},
\eqref{holder}, and \eqref{xnm-1+}, we infer
\begin{align*}
\sum_{i=1}^{m-1}|I_{233}^i|
&\le C\tau(1+\|X_0\|_{1+\gamma}^{4q(q+1)^2+1}).
\end{align*}
For the term $I_{234}^i$, by \eqref{Du} with $\alpha=\varrho$, H\"older inequality, and \eqref{D2F-d} with $(\mu_1,\mu_2)=(0,\chi)$, we obtain
\begin{align*}
|I_{234}^i|
&\le C \int_{t_i}^{t_{i+1}}\phi_\varrho(t_m-t) (1+\sup_{0\le\theta\le1} \|Z_\theta^i\|_{L_\omega^{4(q-1)}\dot H^{1+\gamma}}^{q-1})\\
&\quad\times
\|\Psi_{t_m-t}^{(1)}(\hat X^N)\|_{L_\omega^4} \|\hat X^N-X_i^N\|_{L_\omega^4} \|\hat X^N-X_i^N\|_{L_\omega^4\dot H^\chi}\dd t .
\end{align*}
By \eqref{holder} with $p=4$, $\nu=0$ and $\nu=\chi$, respectively,
\begin{align*}
\|\hat X^N-X_i^N\|_{L_\omega^4} \|\hat X^N-X_i^N\|_{L_\omega^4\dot H^\chi}
&\le C(t-t_i)(1+\|X_0\|_{1+\gamma}^{2(q+1)^2}).
\end{align*}
Combining this with \eqref{phi}, \eqref{xhat-est}, \eqref{xnm-1+}, and \eqref{ele-int}, we get
\begin{align*}
\sum_{i=1}^{m-1}|I_{234}^i|
&\le C\tau(1+\|X_0\|_{1+\gamma}^{4q(q+1)^2}).
\end{align*}
Combining the estimates for $I_{231}^i$--$I_{234}^i$, we obtain
\begin{align*}
\sum_{i=1}^{m-1}|I_{23}^i|
&\le C\tau^{\frac12+\varrho} (1+\|X_0\|_{1+\gamma}^{4q(q+1)^2+1}).
\end{align*}
Together with \eqref{i21}, \eqref{i22}, and \eqref{i24}, this proves \eqref{I2-est}.
\end{proof}

\begin{lm} \label{lm-I3} 
Let Assumptions \ref{ap-f}-\ref{ap-G} hold with $\lambda_1 > K_3 +\Lambda_{\rm w}$. 
There exist constants $C>0$ and $\tau_{\max}\in(0,1)$ such that for any $\tau\in(0,\tau_{\max}]$,
\begin{align} \label{I3-est}
  \sup_{m \ge 2}   \sum_{i=1}^{m-1} |I_3^i| \le C
    (\lambda_N^{-\frac{1+\gamma}{2}-\varrho} + \tau^{\frac12+\varrho} )
    (1+\|X_0\|_{1+\gamma}^{(8q-1)(q+1)^2}).
\end{align}
\end{lm}

\begin{proof}
We decompose $I_3^i$ as $I_3^i=:I_{31}^i+I_{32}^i$, where
\begin{align*}
I_{31}^i
&:=\frac12\ee \int_{t_i}^{t_{i+1}} \sum_{l \in \nn_+} \Big[ D^2U_{t_m-t}^M(\hat X^N)( S_{N,\tau}\PP_NG(X_i^N)\widetilde g_l, S_{N,\tau}\PP_NG(X_i^N)\widetilde g_l) \\
&\qquad - D^2U_{t_m-t}^M(\hat X^N) ( \PP_MG(X_i^N)\widetilde g_l, \PP_MG(X_i^N)\widetilde g_l )\Big]\dd t,\\
I_{32}^i
&:=-\frac12\ee \int_{t_i}^{t_{i+1}} \sum_{l \in \nn_+} \Big[ D^2U_{t_m-t}^M(\hat X^N) ( \PP_MG(\hat X^N)\widetilde g_l, \PP_MG(\hat X^N)\widetilde g_l ) \\
&\qquad - D^2U_{t_m-t}^M(\hat X^N ) ( \PP_MG(X_i^N)\widetilde g_l, \PP_MG(X_i^N)\widetilde g_l ) \Big]\dd t .
\end{align*}
By a further decomposition of $I_{31}^i$, we write
$I_{31}^i=:I_{311}^i+I_{312}^i$, where
\begin{align*}
I_{311}^i &:=\frac12\sum_{l \in \nn_+}\ee \int_{t_i}^{t_{i+1}} D^2U_{t_m-t}^M(\hat X^N) \Big( \PP_N(S_{N,\tau} - {\rm Id} )G(X_i^N)\widetilde g_l, \\
&\qquad\qquad\qquad\qquad
(S_{N,\tau}\PP_N+\PP_M)G(X_i^N)\widetilde g_l \Big)\dd t, \\
I_{312}^i &:=\frac12\sum_{l \in \nn_+}\ee \int_{t_i}^{t_{i+1}} D^2U_{t_m-t}^M(\hat X^N) \Big( \PP_M(\PP_N-{\rm Id})G(X_i^N)\widetilde g_l, \\
&\qquad\qquad\qquad\qquad
(S_{N,\tau}\PP_N+\PP_M)G(X_i^N)\widetilde g_l \Big)\dd t.
\end{align*}
By \eqref{D2u-mix} with $\vartheta=\varrho$ and H\"older inequality, we estimate $I_{31}^i$ as
\begin{align*}
|I_{31}^i|
&\le C \int_{t_i}^{t_{i+1}}\phi_\varrho(t_m-t) \ee[\|(S_{N,\tau}\PP_N+\PP_M)G(X_i^N)\|_{\LL_2^\chi} \Psi_{t_m-t}^{(2)}(\hat X^N)\\
&\qquad\times (\|(-A)^{-\varrho} (S_{N,\tau}-{\rm Id})G(X_i^N)\|_{\LL_2^0} + \|(-A)^{-\varrho} (\PP_N-{\rm Id})G(X_i^N)\|_{\LL_2^0})] \dd t\\
&\le C\|(-A)^{-(1+\gamma)/2-\varrho}(S_{N,\tau}-{\rm Id})\|_{\LL} \int_{t_i}^{t_{i+1}}\phi_\varrho(t_m-t) \|\Psi_{t_m-t}^{(2)}(\hat X^N)\|_{L_\omega^2}\\
&\qquad\times \|G(X_i^N)\|_{L_\omega^4\LL_2^{1+\gamma}}^2\dd t .
\end{align*}
%
Using \eqref{pn}, \eqref{sn}, \eqref{G1}, \eqref{G2}, \eqref{phi}, \eqref{xhat-est}, and \eqref{xnm-1+}, we obtain
\begin{align}
\sum_{i=1}^{m-1}|I_{31}^i|
&\leq C ( \tau^{1/2+\varrho} +\lambda_N^{-(1+\gamma)/2-\varrho} ) (1+\|X_0\|_{1+\gamma}^{4q(q+1)^2}).
\label{I31}
\end{align}

Next, we turn to $I_{32}^i$. Since $D^2U_{t_m-t}^M(\hat X^N)$ is
symmetric, the identity $B(u,u)-B(v,v)=B(u-v,u-v)+2B(u-v,v),~ u,v \in V_M$  for symmetric bilinear forms $B$ on $V_M$ gives
\begin{align*}
I_{32}^i
&= -\frac12\ee \int_{t_i}^{t_{i+1}}\sum_{l \in \nn_+}
D^2U_{t_m-t}^M(\hat X^N)
((\PP_MG(\hat X^N)-\PP_MG(X_i^N))\widetilde g_l,\\
&\qquad
(\PP_MG(\hat X^N)-\PP_MG(X_i^N))\widetilde g_l)\dd t\\
&\quad -\ee \int_{t_i}^{t_{i+1}}\sum_{l \in \nn_+}
D^2U_{t_m-t}^M(\hat X^N)
((\PP_MG(\hat X^N)-\PP_MG(X_i^N))\widetilde g_l,
\PP_MG(X_i^N)\widetilde g_l)\dd t\\
&=:I_{321}^i+I_{322}^i .
\end{align*}
For $I_{321}^i$, by \eqref{D2u} with $(\vartheta_1,\vartheta_2)=(0,0)$, \eqref{G-lip}, and H\"older inequality, we have
\begin{align*}
|I_{321}^i| 
& \le C \int_{t_i}^{t_{i+1}}\phi_{\kappa_{d,\varrho}+\delta}(t_m-t) \|\Psi_{t_m-t}^{(2)}(\hat X^N)\|_{L_\omega^2}
\|\hat X^N-X_i^N\|_{L_\omega^4 L_\xi^2}^2
\dd t.
\end{align*}
Using \eqref{phi}, \eqref{holder}, and \eqref{xhat-est}, we obtain
\begin{align}\label{i321}
\sum_{i=1}^{m-1}|I_{321}^i|
&\le C\tau(1+\|X_0\|_{1+\gamma}^{4q(q+1)^2}).
\end{align}
For $I_{322}^i$, we first add and subtract the term with
$D^2U_{t_m-t}^M(X_i^N)$. This gives
\begin{align*}
I_{322}^i
&= -\sum_{l \in \nn_+}\ee \int_{t_i}^{t_{i+1}}D^2U_{t_m-t}^M(X_i^N) ( \PP_MG(\hat X^N)\widetilde g_l -\PP_MG(X_i^N)\widetilde g_l, \\
&\qquad\qquad
\PP_MG(X_i^N)\widetilde g_l )\dd t \\
&\quad
-\sum_{l \in \nn_+}\ee \int_{t_i}^{t_{i+1}} ( D^2U_{t_m-t}^M(\hat X^N) - D^2U_{t_m-t}^M(X_i^N) )  \\
&\qquad\qquad
( \PP_MG(\hat X^N)\widetilde g_l -\PP_MG(X_i^N)\widetilde g_l, \PP_MG(X_i^N)\widetilde g_l )\dd t
=: R_1^i+R_2^i .
\end{align*}
To simplify the notation, for $v \in H$ we introduce
\begin{align*}
	\mathcal H_i^t(v)
:= \sum_{l \in \mathbb N^+}D^2U^M_{t_m-t}(X_i^N)(\PP_MDG(X_i^N)(v)\widetilde g_l, \PP_MG(X_i^N)\widetilde g_l).
\end{align*}
For the term $R_1^i$, by Taylor formula and \eqref{xhat}, we have
\begin{align*}
R_1^i
&= -\ee \int_{t_i}^{t_{i+1}}(t-t_i) \mathcal H_i^t(A_NS_{N,\tau}X_i^N)\dd t 
-\ee \int_{t_i}^{t_{i+1}}(t-t_i) \mathcal H_i^t(S_{N,\tau}\PP_NF_\tau(X_i^N))\dd t \\
&\quad
-\ee \int_{t_i}^{t_{i+1}}  \mathcal H_i^t \Big( \int_{t_i}^{t}S_{N,\tau}\PP_NG(X_i^N)\dd W(s)\Big ) \dd t 
-\ee \int_{t_i}^{t_{i+1}} \sum_{l \in \nn_+} D^2U_{t_m-t}^M(X_i^N) \\
&\qquad \Big( \PP_M \int_0^1 D^2G(Z_\theta^i) (\hat X^N-X_i^N,\hat X^N-X_i^N) (1-\theta)\dd\theta  \widetilde g_l, 
\PP_MG(X_i^N)\widetilde g_l \Big)\dd t \\
& =:\sum_{j=1}^4R_{1j}^i.
\end{align*}
For the term $R_{11}^i$, we decompose it in the same spirit as $I_{231}^i$.
Using \eqref{full-sum}, we have
\begin{align*}
R_{11}^i&=
-\ee \int_{t_i}^{t_{i+1}} (t-t_i) \mathcal H_i^t(A_NS_{N,\tau}^{i+1}X^N_0)\dd t\\
&\quad
-\tau \ee \int_{t_i}^{t_{i+1}} (t-t_i) \mathcal H_i^t \Big(\sum_{j=0}^{i-1}A_NS_{N,\tau}^{i+1-j}\PP_NF_\tau(X_j^N) \Big)\dd t\\
&\quad 
-\ee \int_{t_i}^{t_{i+1}} (t-t_i) \mathcal H_i^t\Big( \int_0^{t_i}\Gamma_i(s)\dd W(s)\Big)\dd t
=:\sum_{j=1}^3R_{11j}^i.
\end{align*}
By \eqref{D2u-mix} with $\vartheta=0$, \eqref{DG}, H\"older inequality, and \eqref{smo-pn}, we obtain
\begin{align*}
|R_{111}^i|
& \le C\tau \int_{t_i}^{t_{i+1}} e^{-c(t_m-t)} \|(-A_N)^\varrho S_{N,\tau}^{i}\|_{\LL} \|(-A_N)^{\frac12-\varrho}S_{N,\tau}\|_{\LL} \\
&\qquad\times
\|\Psi_{t_m-t}^{(2)}(X_i^N)\|_{L_\omega^2} \|X^N_0\|_{L_\omega^4\dot H^1} \|G(X_i^N)\|_{L_\omega^4\LL_2^1}\dd t \\
& \le C\tau^{\frac12+\varrho} (1+\|X_0\|_{1+\gamma}^{4q(q+1)^2}) \int_{t_i}^{t_{i+1}}e^{-c(t_m-t)} t_i^{-\varrho}e^{-ct_i}\dd t .
\end{align*}
For $R_{112}^i$, using \eqref{D2u}, \eqref{DG}, H\"older inequality, \eqref{smo-pn}, and \eqref{disc-sum}, we get
\begin{align*}
|R_{112}^i|
& \le C\tau^2 \int_{t_i}^{t_{i+1}} \phi_{\kappa_{d,\varrho}+\delta}(t_m-t) \sum_{j=0}^{i-1} \|(-A_N)^{1/2+\varrho}S_{N,\tau}^{i-j}\|_{\LL} \|(-A_N)^{1/2-\varrho}S_{N,\tau}\|_{\LL} \\
&\qquad\times \|\Psi_{t_m-t}^{(2)}(X_i^N)\|_{L_\omega^2} \|\PP_NF_\tau(X_j^N)\|_{L_\omega^4} \|G(X_i^N)\|_{L_\omega^4\LL_2^0}\,\dd t \\
& \le C\tau^{1/2+\varrho} (1+\|X_0\|_{1+\gamma}^{4q(q+1)^2}) \int_{t_i}^{t_{i+1}} \phi_{\kappa_{d,\varrho}+\delta}(t_m-t)\,\dd t .
\end{align*}
For $R_{113}^i$, by the Malliavin IBP formula \eqref{ibp} and the chain rule \eqref{chain}, we obtain
\begin{align*}
R_{113}^i
&= - \int_{t_i}^{t_{i+1}}(t-t_i)
 \int_0^{t_i}\sum_{l,l' \in \nn_+}
\ee\Big[ D^3U_{t_m-t}^M(X_i^N)
\Big(\DD_s^l X_i^N, \\
&\qquad\qquad
\PP_MDG(X_i^N)(\Gamma_i^l(s))\widetilde g_{l'},
\PP_MG(X_i^N)\widetilde g_{l'} \Big) \Big]\dd s\dd t \\
&\quad
- \int_{t_i}^{t_{i+1}}(t-t_i)
 \int_0^{t_i}\sum_{l,l' \in \nn_+}
\ee\Big[ D^2U_{t_m-t}^M(X_i^N)
\Big(
\PP_MD^2G(X_i^N)(\DD_s^l X_i^N, \Gamma_i^l(s))\widetilde g_{l'}, \\
&\qquad\qquad
\PP_MG(X_i^N)\widetilde g_{l'} \Big) \Big]\dd s\dd t \\
&\quad 
- \int_{t_i}^{t_{i+1}}(t-t_i)
 \int_0^{t_i}\sum_{l,l' \in \nn_+}
\ee \Big[ D^2U_{t_m-t}^M(X_i^N)
\Big( \PP_MDG(X_i^N)(\Gamma_i^l(s))\widetilde g_{l'}, \\
&\qquad\qquad
\PP_MDG(X_i^N)(\DD_s^l X_i^N)\widetilde g_{l'} \Big) \Big]\dd s\dd t
=:\sum_{j=1}^3 R_{113j}^i .
\end{align*}
For $R_{1131}^i$, using \eqref{D3u}, H\"older inequality, and \eqref{DG}, and assigning the $\dot H^\chi$-direction to
$G(X_i^N)\widetilde g_{l'}$, we obtain
\begin{align*}
|R_{1131}^i|
& \le C \int_{t_i}^{t_{i+1}}(t-t_i)
\phi_{\kappa_{d,\varrho,\chi}+\delta}(t_m-t)
\|\Psi_{t_m-t}^{(3)}(X_i^N)\|_{L_\omega^2}  \\
&\quad\times
 \int_0^{t_i}
\|\DD_sX_i^N\|_{L_\omega^4\LL_2^0}
\|\Gamma_i(s)\|_{L_\omega^8\LL_2^0}
\|G(X_i^N)\|_{L_\omega^8\LL_2^\chi}
\dd s\dd t .
\end{align*}
Combining this with \eqref{phi}, \eqref{Gamma_L2+},
\eqref{DM-pointwise-p}, \eqref{G1}, \eqref{xnm-1+}, and \eqref{disc}, we infer
\begin{align*}
|R_{1131}^i| 
& \le C\tau(1+\|X_0\|_{1+\gamma}^{(8q-1)(q+1)^2}) \int_{t_i}^{t_{i+1}} \phi_{\kappa_{d,\varrho,\chi}+\delta}(t_m-t)\dd t .
\end{align*}
Similarly, by \eqref{D2u}, \eqref{D2G}, \eqref{DG}, \eqref{DM-pointwise-p}, \eqref{Gamma_L2+}, \eqref{xnm-1+}, and \eqref{disc}, we obtain
\begin{align*}
|R_{1132}^i|+|R_{1133}^i|
& \le C\tau(1+\|X_0\|_{1+\gamma}^{4q(q+1)^2}) \int_{t_i}^{t_{i+1}} \phi_{\kappa_{d,\varrho}+\delta}(t_m-t) \dd t .
\end{align*}
Combining the above estimates with the estimates for $R_{111}^i$ and $R_{112}^i$, we have
\begin{align*}
\sum_{i=1}^{m-1}|R_{11}^i|
&\le C\tau^{\frac12+\varrho}
(1+\|X_0\|_{1+\gamma}^{(8q-1)(q+1)^2}).
\end{align*}
Similarly, for $R_{12}^i$, by \eqref{D2u-mix} with $\vartheta=0$, H\"older inequality, \eqref{phi}, \eqref{DG}, \eqref{G1}, \eqref{f-grow}, and \eqref{xnm-1+},
we obtain
\begin{align*}
\sum_{i=1}^{m-1}|R_{12}^i|
& \le C\tau \sum_{i=1}^{m-1} \int_{t_i}^{t_{i+1}} e^{-c(t_m-t)} \ee[\Psi_{t_m-t}^{(2)}(X_i^N) \|S_{N,\tau}\PP_NF_\tau(X_i^N)\| \|G(X_i^N)\|_{\LL_2^\chi}]\dd t \\
& \le C\tau (1+\|X_0\|_{1+\gamma}^{4q(q+1)^2}).
\end{align*}
For each $t \in [t_i,t_{i+1}]$, the random linear functional
$\mathcal H_i^t$ and $S_{N,\tau}\PP_NG(X_i^N)$ are
$\mathcal F_{t_i}$-measurable. Hence, by the martingale property,
\begin{align*}
R_{13}^i
&= - \int_{t_i}^{t_{i+1}} \ee \Big[ \mathcal H_i^t\Big( \int_{t_i}^{t}S_{N,\tau}\PP_NG(X_i^N) \dd W(s) \Big)\Big]\dd t  =0.
\end{align*}
Moreover, by \eqref{D2u}, H\"older inequality, \eqref{phi}, \eqref{D2G}, \eqref{G-lip}, \eqref{holder}, and \eqref{xnm-1+},
\begin{align*}
\sum_{i=1}^{m-1}|R_{14}^i|
&\le C\tau(1+\|X_0\|_{1+\gamma}^{4q(q+1)^2}).
\end{align*}
Together with the estimates for $R_{11}^i$ and $R_{12}^i$, this gives
\begin{align}\label{R1-est}
\sum_{i=1}^{m-1}|R_1^i|
&\le C\tau^{\frac12+\varrho}
(1+\|X_0\|_{1+\gamma}^{(8q-1)(q+1)^2}).
\end{align}
We next estimate $R_2^i$. By the mean value theorem, \eqref{D3u}, H\"older inequality, \eqref{G-lip}, and using \eqref{phi}, \eqref{holder}, \eqref{xhat-est}, and \eqref{xnm-1+}, we infer
\begin{align}\label{R2-est}
\sum_{i=1}^{m-1}|R_2^i|
&\le C\tau(1+\|X_0\|_{1+\gamma}^{(8q-1)(q+1)^2}).
\end{align}
Combining \eqref{R1-est} and \eqref{R2-est}, we obtain
\begin{align*}
\sum_{i=1}^{m-1}|I_{322}^i|
&\le C\tau^{\frac12+\varrho}
(1+\|X_0\|_{1+\gamma}^{(8q-1)(q+1)^2}).
\end{align*}
Together with \eqref{i321}, this gives
\begin{align}\label{I32}
\sum_{i=1}^{m-1}|I_{32}^i|
&\le C\tau^{\frac12+\varrho}
(1+\|X_0\|_{1+\gamma}^{(8q-1)(q+1)^2}).
\end{align}
Combining \eqref{I31} and \eqref{I32}, we conclude \eqref{I3-est}. 
\end{proof}

With the preceding estimates established, we are ready to prove Theorem \ref{tm-weak}.

\begin{proof}[Proof of Theorem \ref{tm-weak}]
Combining Lemmas \ref{lm-u}-\ref{lm-I3}, we conclude \eqref{weak}. 
\end{proof}

\section{Numerical Experiments}\label{sec5}
In this section, we design numerical tests to illustrate the ergodicity, as well as the sharpness and time-independence of the weak convergence rate in Theorem \ref{tm-weak} for SACE on $\mathcal O=(0,1)^2$ with homogeneous DBC:
\begin{align*}
f(\xi)=\varepsilon^{-2}(\xi-\xi^3),\qquad
g(\xi)=\sigma_1+\sigma_2\xi,\qquad
\varepsilon=0.35,\quad (\sigma_1,\sigma_2)=(2,0.2).
\end{align*}
The noise is expanded in the sine basis
\begin{align*}
e_{k,\ell}(\xi_1,\xi_2)
=2\sin(k\pi\xi_1)\sin(\ell\pi\xi_2),\qquad 
q_{k,\ell}=(1+k^2+\ell^2)^{-3},
\quad k,\ell \in \nn_+ .
\end{align*}
We use the tensor-product spectral space $V_N={\rm span}\{e_{k,\ell}:1\le k,\ell\le N\}$.
In the simulations, the $Q$-Wiener process is truncated as
$\sum_{1\le k,\ell\le N_Q}q_{k,\ell}^{1/2}e_{k,\ell}\beta_{k,\ell}(t)$, 
where $\{\beta_{k,\ell}\}$ are independent 1D Brownian motions. 

The parameters satisfy the dissipativity condition in Theorem~\ref{tm-weak}. Indeed, $\lambda_1=2\pi^2$, and the polynomial exponent in Assumption~\ref{ap-f} is $q=2$. Moreover,
\begin{align*}
(1+\tau\xi^4)f'(\xi)-2\tau\xi^3f(\xi)
&=\varepsilon^{-2}(1-3\xi^2-\tau\xi^4-\tau\xi^6)
\le\varepsilon^{-2}(1+\tau\xi^4)^{3/2},
\end{align*}
so we may take $K_3=\varepsilon^{-2}\approx8.1633$.
For the untruncated covariance operator, let
\begin{align*}
S_0:=\sum_{k,\ell\ge1}(1+k^2+\ell^2)^{-3}\approx0.05124,\quad 
S_2:=\sum_{k,\ell\ge1} \frac{k^2+\ell^2}{(1+k^2+\ell^2)^3}\approx0.19857.
\end{align*}
From the uniform bound $\|e_{k,\ell}\|_{L^ \infty}\le2$, the product rule, the Poincar\'e inequality \eqref{poin}, and Young inequality, we obtain
\begin{align*}
\|DG(u)h\|_{\LL_2^0}^2 \le4\sigma_2^2S_0\|h\|^2,\quad 
\|G(v)\|_{\LL_2^1}^2 \le K_5+1.06\sigma_2^2 (2\sqrt{S_0}+\sqrt{2S_2})^2\|v\|_1^2.
\end{align*}
Thus, we may choose $K_4^2=4\sigma_2^2S_0\approx0.00820$ and $K_6=0.05$.
Moreover, for every $\gamma\in(0,1)$, with $\lambda_{k,\ell}=\pi^2(k^2+\ell^2)$, we have $\sum_{k,\ell\ge1}q_{k,\ell}\lambda_{k,\ell}^{1+\gamma}\le C\sum_{k,\ell\ge1}(1+k^2+\ell^2)^{-2+\gamma}<\infty$. Together with the algebra property of $\dot H^{1+\gamma}$, this gives $\|G(z)\|_{\LL_2^{1+\gamma}}\le C(1+\|z\|_{1+\gamma})$, $z\in\dot H^{1+\gamma}$, so \eqref{G2} holds.
Consequently,
\begin{align*}
K_3+\Lambda_{\rm w}
=\max \Big\{ K_3+\frac{13}{2}K_4^2, K_3+\frac{251}{2}K_6 \Big\}\le14.44<2\pi^2=\lambda_1.
\end{align*}

\emph{UIT weak convergence rate.}
We use the bounded observables
\begin{align*}
\varphi_1(u)=\sin(\|u\|),\qquad
\varphi_2(u)=\exp(-\|u\|^2),\qquad
\varphi_3(u)=\cos(\|u\|^2).
\end{align*}
For fixed signs $\eta_{k,\ell} \in \{-1,1\}$, $M \in \nn_+$, and $\alpha,p>0$, set
\begin{align*}
u_{\alpha,p}^{M}
=\sum_{k,\ell=1}^{M}
\frac{\eta_{k,\ell}e_{k,\ell}}
{(k^2+\ell^2)^\alpha\log^p(e+\sqrt{k^2+\ell^2})},
\qquad
\bar u_{\alpha,p}^{M}
=\frac{u_{\alpha,p}^{M}}{\|u_{\alpha,p}^{M}\|},
\end{align*}
and take
$X_{0,\mathrm{time}}=0.8e_{1,1}+\bar u_{1.05,1.2}^{64}$.
We fix $N=32$, $N_Q=64$, $\mathcal T=\{1,10,50\}$, $\tau_{\rm ref}=2^{-12}$. The tested time steps are  $\tau \in \{2^{-6},2^{-7},2^{-8},2^{-9}\}$. 
The coarse and reference paths use the same truncated Wiener path, with coarse increments obtained by summing the corresponding fine increments. For $T \in \mathcal T$ and an observable $\varphi$, set
\begin{align*}
    E_T^\varphi(\tau)
    = \Big \lvert
    K_{\rm mc}^{-1}\sum_{r=1}^{K_{\rm mc}}
    [\varphi(X_{T/\tau}^{N,\tau,(r)})-\varphi(X_{\rm ref}^{N,\tau_{\rm ref},(r)}(T))]
    \Big\rvert.
\end{align*}
The temporal convergence figure has one panel for each $T \in \mathcal T$ and shows the three errors $E_T^{\varphi_\ell}(\tau)$, $\ell=1,2,3$. The constant in the reference line of the $T$-panel is
\begin{align*}
    C_{T,\varrho_*}^{\rm time}
    = \max_\tau
    \frac{\max_{\ell=1,2,3}E_T^{\varphi_\ell}(\tau)}
    {(1+T^{-\varrho_*})\,\tau^{\varrho_*}},
    \quad \varrho_* \in (0,1).
\end{align*}
With $K_{\rm mc}=500$, the fitted slopes are compared with the rate
$\tau^\rho$, $\rho<1$, in Theorem \ref{tm-weak}.
Figure~\ref{fig:time} shows that the temporal weak errors decay almost linearly in $\tau$ and remain stable for $T=1,10,50$.

\begin{figure}[t]
\centering
 \includegraphics[width=\textwidth]{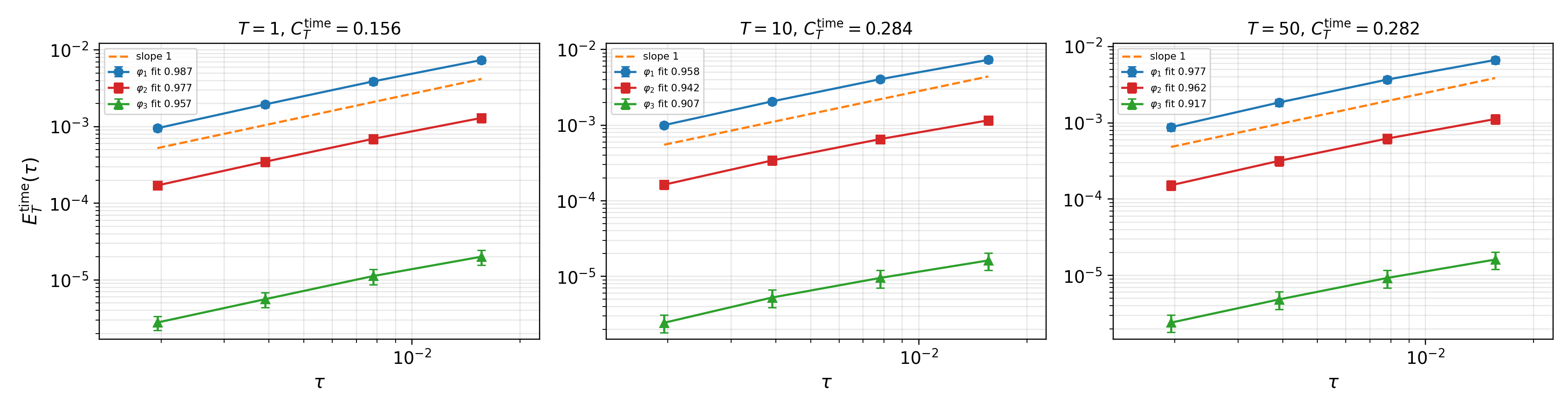}
\caption{Temporal weak errors.}
\label{fig:time}
\end{figure}

\emph{Ergodicity test.}
We take $N=32$, $N_Q=64$, $\tau=2^{-8}$, $T=50$, and $K_{\rm mc}=1000$, with
$X_0^{(1)}=1.4e_{1,1}$,
$X_0^{(2)}=-1.4e_{1,1}$, and
$X_0^{(3)}=(e_{2,1}+e_{1,2})/\sqrt2$.
For each Monte--Carlo path, all three initial states are driven by the same truncated Wiener path. We monitor
\begin{align*}
\varphi_1(u)&=\tanh(\langle u,e_{1,1}\rangle),\quad
\varphi_2(u)=\tanh(\langle u,e_{2,1}+e_{1,2}\rangle/\sqrt2),\quad
\varphi_3(u)=1-e^{-\|u\|^2},
\end{align*}
and
\begin{align*}
\ee_{\varphi_\ell}^{(r)}(t_j)
&=K_{\rm mc}^{-1}\sum_{m=1}^{K_{\rm mc}}  \varphi_\ell(X_j^{N,\tau,(m)}(X_0^{(r)})).
\end{align*}
Figure~\ref{fig:ergodic} shows the merging of the ensemble means $\ee_{\varphi_\ell}^{(r)}$ from the three initial states, illustrating the loss of memory of initial data and supporting the ergodic behavior related to Corollary \ref{cor-erg}.

\begin{figure}[t]
\centering
 \includegraphics[width=\textwidth]{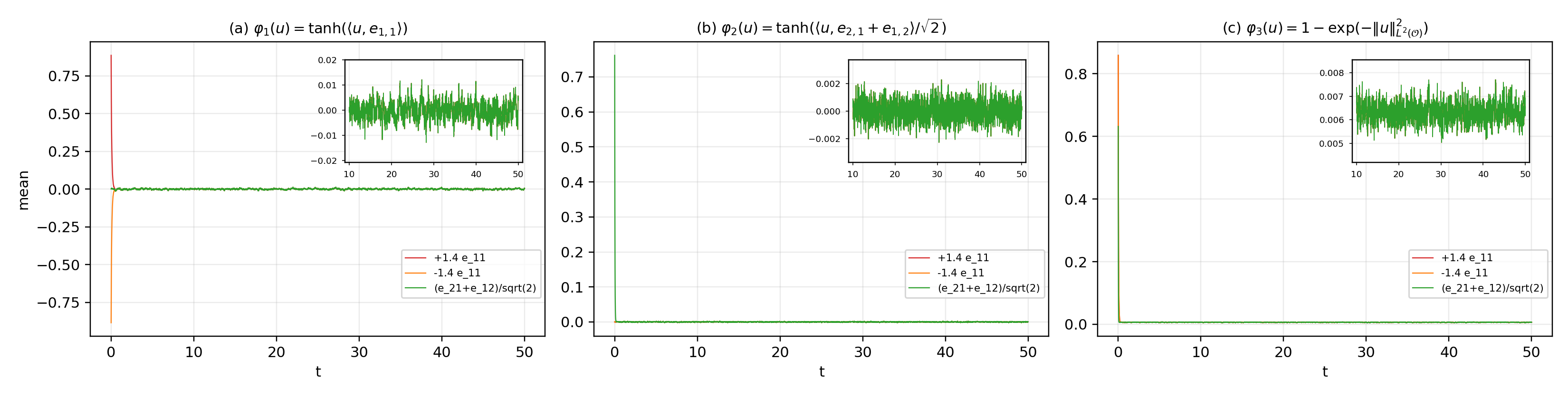}
\caption{Ergodicity test.}
\label{fig:ergodic}
\end{figure}

\section*{Appendix}
\newcounter{appsec}
\renewcommand{\theappsec}{\Alph{appsec}}
\providecommand{\theHappsec}{}
\renewcommand{\theHappsec}{appendix.\arabic{appsec}}
\makeatletter
\@addtoreset{equation}{appsec}
\makeatother
\renewcommand{\theequation}{\theappsec.\arabic{equation}}
\providecommand{\theHequation}{}
\renewcommand{\theHequation}{appendix.\arabic{appsec}.\arabic{equation}}

\refstepcounter{appsec}
\subsection*{Appendix \theappsec. Proof of the estimates in Remark~\ref{rk-F-bounds}}\label{app-A}
The estimates \eqref{ftau'}--\eqref{f-ftau} follow directly from
\cite[Lemma~4.2]{LS25}. By Assumption~\ref{ap-f} and the definition \eqref{f-tau} of $f_\tau$, one has $|f_\tau(\xi)|\le C(1+|\xi|^{q+1})$ and $|f_\tau'(\xi)|\le C(1+|\xi|^q)$, $\xi \in \rr$. 
Hence, the chain rule yields \eqref{F-bounds-0} and \eqref{F-bounds-1}.
It remains to prove \eqref{D2F-d} and \eqref{D3F-d}.
By the duality between $\dot H^{-2\varrho}$ and $\dot H^{2\varrho}$, it is enough to estimate $\langle D^2F(w)(u_1,u_2),z\rangle$ for $z \in \dot H^{2\varrho}$.
By the definition of the Nemytskii operator and Assumption~\ref{ap-f},
\begin{align*}
|\langle D^2F(w)(u_1,u_2),z\rangle|
&\le C(1+\|w\|_{L_\xi^ \infty}^{q-1}) \int_{\mathcal O}|u_1(\xi)||u_2(\xi)||z(\xi)|\dd \xi .
\end{align*}
Choose $r_1,r_2,r_3 \in [2, \infty]$ such that
\begin{align*}
&\frac1{r_1}+\frac1{r_2}+\frac1{r_3}=1,\qquad
\frac1{r_j}\ge \frac12-\frac{\mu_j}{d},~j=1,2,\qquad
\frac1{r_3}\ge \frac12-\frac{2\varrho}{d} .
\end{align*}
Such a choice is possible since $\mu_1+\mu_2+2\varrho>d/2$.
For $d=1,2,3$, the Sobolev embeddings $\dot H^{\mu_j}\hookrightarrow L_\xi^{r_j}$, $j=1,2$, and $\dot H^{2\varrho}\hookrightarrow L_\xi^{r_3}$ hold. Hence, by H\"older inequality,
\begin{align*}
|\langle D^2F(w)(u_1,u_2),z\rangle|
&\le C(1+\|w\|_{L_\xi^ \infty}^{q-1}) \|u_1\|_{\mu_1}\|u_2\|_{\mu_2}\|z\|_{2\varrho}.
\end{align*}
Taking the supremum over all $z \in \dot H^{2\varrho}$ with $\|z\|_{2\varrho}\le1$ gives \eqref{D2F-d}.
Similarly,  
\begin{align*}
|\langle D^3F(w)(u_1,u_2,u_3),z\rangle|
&\le C(1+\|w\|_{L_\xi^ \infty}^{q-2})
 \int_{\mathcal O}|u_1(\xi)||u_2(\xi)||u_3(\xi)||z(\xi)|\dd \xi,
~ z \in \dot H^{2\varrho}.
\end{align*}
Choose $r_1,r_2,r_3,r_4 \in [2, \infty]$ such that
\begin{align*}
&\sum_{j=1}^4\frac1{r_j}=1,\qquad
\frac1{r_j}\ge \frac12-\frac{\nu_j}{d},~j=1,2,3,\qquad
\frac1{r_4}\ge\max\{ \frac12-\frac{2\varrho}{d},0\} .
\end{align*}
Such a choice is possible since $\nu_1+\nu_2+\nu_3+\min\{2\varrho,\frac d 2\}>d$.
Hence, by H\"older inequality and the Sobolev embeddings,
\begin{align*}
|\langle D^3F(w)(u_1,u_2,u_3),z\rangle|
&\le C(1+\|w\|_{L_\xi^ \infty}^{q-2})
\prod_{j=1}^3\|u_j\|_{\nu_j}\|z\|_{2\varrho}.
\end{align*}
Taking the supremum over $\|z\|_{2\varrho}\le1$ gives \eqref{D3F-d}.

\refstepcounter{appsec}
\subsection*{Appendix \theappsec. Smoothing estimates}

There exist positive constants $C$ and $c \in (0,\lambda_1)$ such that 
\begin{align}
\|(-A)^\beta S(t)\|_{\LL} \le C t^{-\beta}e^{-c t}, \quad & \beta>0,~ t \in \rr_+, \label{smo}\\ 
\|(-A)^\beta S_{N,\tau}^j \PP_N\|_{\LL} \le C t_j^{-\beta}e^{-ct_j}, \quad & \beta \in [0, 1],~ j \in \nn_+, \label{smo-pn}\\ 
\|(-A)^{-\beta} ({\rm Id} - \PP_N)\|_{\LL} \le \lambda_N^{-\beta}, \quad & \beta>0, \label{pn} \\
\|(-A)^{-\beta}(S_{N,\tau} - {\rm Id} )\|_{\LL} \le C \tau^\beta, \quad & \beta \in [0, 1],  \label{sn} \\
\|(-A_N)^{-\beta}(e^{\tau A_N}-S_{N,\tau})\|_{\LL} \le C\tau^\beta, \quad & \beta \in [0, 1],  \label{sn+} \\
 \int_0^{a_j} \|(-A_N)^{3/2} S_{N,\tau}^{j+1-\ell(s)} \PP_N\|_{\LL}^2\,\dd s \le C, \quad & j \in \nn_+. \label{noise-smo-est}
\end{align}

\begin{proof}
The first four estimates are well-known by expanding the Fourier modes.
By the spectral decomposition of $-A_N$ and the elementary estimate
$\sup_{x>0}x^{-\beta}|e^{-x}-(1+x)^{-1}|< \infty$ for $\beta \in [0,1]$, we obtain
\begin{align*}
\|(-A_N)^{-\beta}(e^{\tau A_N}-S_{N,\tau})\|_{\LL} \le \tau^\beta \sup_{x>0}x^{-\beta} |e^{-x}-(1+x)^{-1} | \le C\tau^\beta .
\end{align*} 
The last one is trivial, if $a_j:=(t_j-1)_+=0$, $j \in \nn_+$. 
Otherwise, for $s \in [0, a_j]$, one has $t_{j+1}-[s]_\tau \ge 1$. 
By spectral decomposition,
\begin{align*}
 \int_0^{a_j}
\|(-A_N)^{3/2}S_{N,\tau}^{j+1-\ell(s)}\PP_N\|_{\LL}^2\,\dd s
& \le C \int_0^{a_j} (t_{j+1}-[s]_\tau)^{-3} (1+\lambda_1\tau)^{-(j+1-\ell(s))} \dd s \\
& \le C\tau\sum_{r=1}^{ \infty}(1+\lambda_1\tau)^{-r}
\le C.
\end{align*}
We note that $\ell(s)$ and $[s]_\tau$ above were introduced in Section~\ref{sec2}.
This shows \eqref{noise-smo-est}.
\end{proof}

\refstepcounter{appsec}
\subsection*{Appendix \theappsec. Discrete convolution estimates}
Let $c>0$, $\eta \in [0,1)$, and $\varepsilon > 0$.
There exists $C>0$ such that 
\begin{align}
& \tau\sum_{\ell=1}^{ \infty}(\ell\tau)^{-1+\varepsilon}e^{-c\ell\tau} \le C, \label{disc-sum}\\
& \sup_{m \in \nn_+} \Big( \int_0^{t_m} + \int_{a_m}^{t_m}\Big) (t_m-[s]_\tau)^{-1+\varepsilon}e^{-c(t_m-[s]_\tau)} \dd s  \le C, \label{disc}\\
&\sum_{i=1}^{m-1} \int_{t_i}^{t_{i+1}} \phi_{\eta}(t_m-t) (1+t_i^{-1+\varepsilon}e^{-ct_i}) \dd t \le C (1+t_m^{-\max\{\eta-\varepsilon,0\}} ) ,\label{ele-int}
\end{align}
where $\phi_{\eta}(\cdot)$ is defined in \eqref{phi-a}.

\begin{proof}
The first bound is a standard sum-integral comparison:
\begin{align*}
\tau\sum_{\ell=1}^{ \infty}(\ell\tau)^{-1+\varepsilon}e^{-c\ell\tau}
\le C \int_0^ \infty s^{-1+\varepsilon}e^{-cs}\,\dd s
\le C< \infty.
\end{align*}
For the next two bounds, we use the piecewise constancy of $[s]_\tau$ to write
\begin{align*}
& \Big( \int_0^{t_m} + \int_{a_m}^{t_m}\Big) (t_m-[s]_\tau)^{-1+\varepsilon}e^{-c(t_m-[s]_\tau)} \dd s \\
& \le\tau\sum_{i=1}^{m}(i\tau)^{-1+\varepsilon}e^{-c i \tau} 
+ \tau\sum_{i=1}^{\min\{m,\lceil 1/\tau\rceil\}}(i \tau)^{-1+\varepsilon}e^{-c i \tau},
\end{align*}
which is bounded by $C (1+ \int_0^ \infty s^{-1+\varepsilon}e^{-cs}\,\dd s )$, proving \eqref{disc}. For $t \in [t_i,t_{i+1}]$, $i\ge1$, we have $t\le2t_i$, so that
$t_i^{-1+\varepsilon}e^{-ct_i}
\le C t^{-1+\varepsilon}e^{-ct/2}$.
Hence, the left-hand side of \eqref{ele-int} is bounded by
$C \int_0^{t_m}
(1+(t_m-t)^{-\eta})e^{-c(t_m-t)}
(1+t^{-1+\varepsilon}e^{-ct/2})\,\dd t$, which is uniformly bounded for $t_m\ge1$. For $0<t_m\le1$, the only
mixed singularity gives, by scaling,
\begin{align*}
 \int_0^{t_m}(t_m-t)^{-\eta}t^{-1+\varepsilon}\,\dd t
\le C t_m^{\varepsilon-\eta}
\le C(1+t_m^{-\max\{\eta-\varepsilon,0\}}),
\end{align*}
while all remaining terms are uniformly bounded. This proves \eqref{ele-int}.
\end{proof}

\refstepcounter{appsec}
\subsection*{Appendix \theappsec. One-step error estimate}
Let $\alpha \in [0,1/2)$ and $X_0 \in \dot H^{1+\gamma}$ with $\gamma \in [0,1)$ satisfying \eqref{gamma}, and assume that Assumptions~\ref{ap-f}-\ref{ap-G} hold. There exist constants $C>0$ and $\tau_{\max}\in(0,1)$ such that for any $M,N \in \nn_+$ with $M>N$ and $\tau \in (0,\tau_{\max}]$,
\begin{align}
\|\PP_NX^M(t_1,X_0^N)-X_1^N\|_{L_\omega^2\dot H^{-2\alpha}}
\le C\tau^{\frac12+\alpha}(1+\|X_0\|_{1+\gamma}^{3q+1}).
\label{one-err}
\end{align}

\begin{proof}
We write $\PP_NX^M(t_1,X_0^N)-X_1^N=\sum_{k=1}^6J_k$, where
\begin{align*}
J_1&:=(\mathrm e^{t_1A_N}-S_{N,\tau})X_0^N,\\
J_2&:= \int_0^{t_1}(\mathrm e^{(t_1-s)A_N}-S_{N,\tau})\PP_NF(X^M(s,X_0^N))\,\dd s,\\
J_3&:= \int_0^{t_1}S_{N,\tau}\PP_N(F(X^M(s,X_0^N))-F(X_0^N))\,\dd s,\\
J_4&:=t_1S_{N,\tau}\PP_N(F(X_0^N)-F_\tau(X_0^N)),\\
J_5&:= \int_0^{t_1}(\mathrm e^{(t_1-s)A_N}-S_{N,\tau})\PP_NG(X^M(s,X_0^N))\,\dd W(s),\\
J_6&:= \int_0^{t_1}S_{N,\tau}\PP_N(G(X^M(s,X_0^N))-G(X_0^N))\,\dd W(s).
\end{align*}
\eqref{sn+} implies $\|J_1\|_{L_\omega^2\dot H^{-2\alpha}} \le C\tau^{\frac12+\alpha}\|X_0\|_1$. 
Using \eqref{sn}, \eqref{f-grow}, \eqref{f-ftau}, \eqref{gamma}, \eqref{G-lip}, the boundedness of $S_{N,\tau}$, It\^o isometry, and \cite[Lemma~4.1]{LS25}, we obtain
\begin{align*}
\sum_{k=2}^5\|J_k\|_{L_\omega^2\dot H^{-2\alpha}}
\le C\tau^{\frac12+\alpha}(1+\|X_0\|_{1+\gamma}^{3q+1}).
\end{align*}
Moreover, a standard argument based on the mild formulation yields the temporal H\"older estimate (see \cite[Theorem~3.2]{LQ21})
\begin{align*}
\|X^M(s,X_0^N)-X_0^N\|_{L_\omega^2L_\xi^2}
\le Cs^{1/2}(1+\|X_0\|_{1+\gamma}^{3q+1}),\qquad s \in [0,t_1].
\end{align*}
Therefore, by It\^o isometry and \eqref{G-lip}, we get
\begin{align*}
\|J_6\|_{L_\omega^2\dot H^{-2\alpha}}^2
&\le C \int_0^{t_1}\|X^M(s,X_0^N)-X_0^N\|_{L_\omega^2L_\xi^2}^2\,\dd s
\le C\tau^2(1+\|X_0\|_{1+\gamma}^{3q+1})^2.
\end{align*}
Summing the above estimates proves \eqref{one-err}.
\end{proof}

\bibliographystyle{abbrv} 
\bibliography{references.bib}
\end{document}